\documentclass{elsarticle}

\usepackage[utf8]{inputenc}
\usepackage[T1]{fontenc}
\usepackage{textcomp}
\usepackage{amsmath,amssymb,amsthm,mathrsfs}
\usepackage{lmodern}
\usepackage[a4paper]{geometry}
\usepackage{graphicx}
\usepackage[usenames,dvipsnames]{xcolor}
\usepackage{microtype}
\usepackage{enumerate}

\usepackage{hyperref}

\newtheoremstyle{theoreme}
 {\topsep} 
 {\topsep} 
 {} 
 {0pt} 
 {\bfseries} 
 {\newline} 
 {3pt} 
 {} 
\newtheoremstyle{proposition}
{\topsep} 
 {\topsep} 
 {} 
 {0pt} 
 {\bfseries} 
 {\newline} 
 {3pt} 
 {} 
 \newtheoremstyle{lemma}
 {\topsep} 
 {\topsep} 
 {} 
 {0pt} 
 {\bfseries} 
 {\newline} 
 {3pt} 
 {} 
\newtheoremstyle{definition}
 {\topsep}
 {\topsep}
 {}
 {0pt}
 {\bfseries}
 {\newline}
 {3pt}
 {}
\newtheoremstyle{remarque}
 {\topsep}
 {\topsep}
 {}
 {0pt}
 {\bfseries}
 {\newline}
 {3pt}
 {} 
\theoremstyle{theoreme}
\newtheorem{Thm}{Theorem}[section]
\newtheorem*{Thm*}{Theorem}
\newtheorem*{Notation}{Notation}
\newtheorem{Ex}[Thm]{Example}
\theoremstyle{lemma}
\newtheorem{Lem}[Thm]{Lemma}
\theoremstyle{proposition}
\newtheorem{Prop}[Thm]{Proposition}
\newtheorem*{Prop*}{Proposition}

\newtheorem{Claim}{Claim}
\newtheorem*{Claim*}{Claim}
\theoremstyle{definition}
\newtheorem{Def}[Thm]{Definition}
\theoremstyle{remarque}
\newtheorem{Rem}[Thm]{Remark}

\newcommand\entiersrel{\mathbb{Z}}

\newcommand\N{\mathbb{N}}

\newcommand\Z{\mathbb{Z}}
\newcommand\Q{\mathbb{Q}}
\newcommand\langue{\mathcal{L}}

\newcommand\etoile{\precsim^{\ast}}
\newcommand\netoile{\precnsim^{\ast}}

\newcommand\qoval{\precsim_{\text{val}}}
\newcommand\eqval{\sim_{\text{val}}}
\newcommand\nqoval{\precnsim_{\text{val}}}
\newcommand\Hahn[1]{\text{H}_{\gamma\in\Gamma}#1_{\gamma}}
\newcommand\supp{\text{supp}}

\newcommand\archqo{\precsim_{\text{arch}}}
\newcommand\narchqo{\precnsim_{\text{arch}}}
\newcommand\archeq{\sim_{\text{arch}}}
\newcommand\archval{v_{\text{arch}}}
\newcommand\preqo{\lessdot}
\newcommand\cprod{\overset{\leftarrow}{\times}}
\begin{document}
 \title{A structure theorem for abelian quasi-ordered groups\tnoteref{t1}\tnoteref{t2}}
 \tnotetext[t1]{All results presented here are part of my PhD. 
   In that regard I thank my supervisors Salma Kuhlmann and Françoise Point for suggesting this project and for 
   the help and support they gave me 
   during its completion, in particular for taking the time to 
   read and discuss my paper with me, as their many comments and suggestions played an essential role in shaping 
   this article into a publishable form.
   I also thank Françoise Delon for her questions which motivated some of my research.}
   \tnotetext[t2]{Throughout this paper, we use the abbreviations 
 q.o  and q.o.a.g which respectively mean quasi-order and quasi-ordered abelian group}
   
 \author{Gabriel Lehéricy}
 \address{Universität Konstanz, Fachbereich Mathematik und Statistik,
Universitätsstraße 10,
78457   Konstanz,  Germany\\
Université Paris 7, Bâtiment Sophie Germain,
8 Place Aurélie Nemours, 75205 PARIS Cedex 13, France}
 \ead{gabriel.lehericy@uni-konstanz.de}
 \begin{abstract}
 
     We introduce a notion of compatible quasi-ordered groups
     which unifies 
    valued and ordered abelian groups. It was proved by S.M. Fakhruddin 
    that a compatible quasi-order on a
    field is always either
    an order or a valuation.
     We show here that the group case is more complicated than the field case and describe the 
     general structure of a compatible quasi-ordered abelian group. We then define a notion of Hahn product 
     of compatible quasi-ordered groups and generalize Hahn's embedding theorem to quasi-ordered groups.      
     We also develop a notion of
     quasi-order-minimality  and establish a connection with C-minimality, thus
     answering a question of F.Delon. Finally, we use compatible quasi-ordered groups to give an example of a
     C-minimal group which is neither an ordered nor a valued group.

 \end{abstract}

  \maketitle
 \section*{Introduction}

     Ordered and valued abelian groups appear naturally in the study of valued fields with additional structures. 
     The theory of exponential fields developed in \cite{SKuhlmann} shows that the value group of an exponential field
     is a contraction group (see \cite{Kuhlmann})
     which contains a lot of information about the field itself. Contraction groups 
    can be seen as a valued group in two
     ways: since it is an ordered group we can endow it with the natural valuation associated to this order;
     moreover, we can associate a valuation to the contraction map $\chi$ by defining
     $v(g):=\vert\chi(-\vert g\vert)\vert$.
     
     Valued groups are also naturally related to the theory of valued differential fields and 
     asymptotic couples, which is the central topic of asymptotic differential algebra.
     In \cite{Rosenlicht1}, \cite{Rosenlicht2}, \cite{Rosenlicht3}, \cite{Rosenlicht4} and \cite{Rosenlicht5}
     Rosenlicht studied Hardy fields and showed that the logarithmic derivative induces a map $\psi$ on 
     $G\backslash\{0\}$, where $G$ is the value group of the field. The pair $(G,\psi)$ is called an asymptotic couple.
     Aschenbrenner and Van Den Dries later gave a model-theoretic approach to asymptotic couples in
     \cite{Asch}, \cite{Aschdries},
     \cite{Aschdries2} and \cite{Aschdries3}. The map $\psi$ becomes a valuation if we extend it to $G$ by setting
     $\psi(0):=\infty$.
     
     Although ordered and valued structures
     are classically treated as different subjects they still bear significant similarities, which is why we are now
     interested in unifying both into a single theory.     
     Fakhruddin made a step in that direction in \cite{Fakhruddin} when he noticed that both orders and valuations could 
     be seen as particular instances of quasi-orders.
     Consider a field $K$ endowed with a quasi-order $\precsim$ satisfying the following axioms,
      where $y\sim z$ is defined as $y\precsim z\precsim y$:
  \begin{itemize}\label{qofields}
    \item[$(Q_1)$] $\forall x(x\sim 0\Rightarrow x=0)$
    \item[$(Q_2)$] $\forall x,y,z(x\precsim y\nsim z\Rightarrow x+z\precsim y+z)$ 
    \item[$(Q_3)$] $\forall x,y,z, (x\precsim y\wedge 0\precsim z)\Rightarrow xz\precsim yz$
   \end{itemize}
     
     The pair $(K,\precsim)$ is what Fakhruddin calls a quasi-ordered field. Here are the two most important results of 
     \cite{Fakhruddin}:
     
     \begin{Prop*}
      If $(K,\leq)$ is an ordered field then it is in particular a quasi-ordered field (i.e $\leq$ satisfies the 
      axioms
      above).    
      If $(K,v)$ is a valued field then $v$ induces a quasi-order on $K$ via $a\precsim b\Leftrightarrow v(a)\geq v(b)$
      and this quasi-order satisfies the axioms above.    
      
     \end{Prop*}

    \begin{Thm*}[Fakhruddin's dichotomy]
    Let $K$ be a field and $\precsim$ a quasi-order on $K$ satisfying the axioms $(Q_1),(Q_2)$ and $(Q_3)$. Then 
     $\precsim$ is either a field order or the quasi-order induced by a field valuation.
   \end{Thm*}

     These results show that the theory of quasi-ordered fields is an excellent way of unifying
     the theory of ordered fields
     with the theory of valued fields.  The main motivation behind this paper
     is to develop a similar theory for abelian groups 
     and give an answer to the following questions:\\
     
     \textit{Does Fakhruddin's dichotomy hold in the case of groups? If not, what is the structure of quasi-ordered
     abelian groups endowed with Fakhruddin's axioms?}\\

     Our hope is that the theory of quasi-orders will be useful in the study of valued groups and thus also
     in the study of valued fields endowed with an operator; example \ref{contreexemple}(c) below already reveals connections between 
     compatible quasi-ordered abelian groups and ordered difference fields. We are particularly interested in model-theoretic aspects
     of quasi-ordered groups, which is why we will also introduce a notion of quasi-order-minimality.
     Notions of minimality have been at the heart of recent developments in model theory; 
     amongst other examples we can mention 
     o-minimality (see \cite{PillaySteinhorn}) and C-minimality (see \cite{Macstein}, \cite{Delon} and \cite{DelonSim}).    
     Our idea is to study 
     quasi-ordered groups whose definable sets are particularly simple. 
     o-minimality is a special case of quasi-order-minimality, but the latter should also give
     interesting results concerning definable sets in valued groups. This
     might help us prove minimality results concerning contraction groups and asymptotic couples.   
     We will also see that the study of compatible quasi-ordered groups will be useful in the classification of 
     C-minimal groups. The notion of 
     C-group and C-minimal groups was introduced in     
     \cite{Macstein} by Macpherson and Steinhorn.
 Delon then 
 generalized the definition of C-group in \cite{Delon} to include ordered groups. In Delon's context, o-minimality and strong minimality both become 
 special cases of C-minimality. 
 Although abelian valued C-minimal groups were completely classified in \cite{DelonSim}, there is still no complete 
 classification of C-minimal groups. This paper shows that the class of compatible quasi-ordered groups constitutes a particularly simple class 
     of C-groups, so that the study of  compatible quasi-ordered groups could be an essential step 
     towards a classification of C-minimal groups.

        We start this paper with a preliminary section in which we recall
        the definitions of valuations and quasi-orders. 
        In Section \ref{qoagsection} we introduce the notion of compatible quasi-ordered abelian group (q.o.a.g), which is the group analog of 
        Fakhruddin's quasi-ordered field. We quickly establish that Fakhruddin's dichotomy fails 
        in the group case by giving an explicit example (Example \ref{contreexemple}). We then focus on describing the structure of a compatible 
        q.o.a.g. We basically show that a compatible q.o.a.g is a ``mix'' of ordered and valued groups, in the sense that 
        a compatible q.o.a.g is an extension of a valued group by an ordered group. More precisely, a compatible q.o.a.g 
        $(G,\precsim)$
         is composed of an ordered part and a part on which the quasi-order behaves like a valuation, the former 
        being an initial segment of $G$. We obtain this result by
         dividing the elements of the group into two categories, respectively
        called o-type and v-type elements. We show that the set of o-type elements is actually an ordered abelian group 
        and  that the quasi-order of $G$ induces a valuation on the quotient of $G$ over the subgroup of o-type elements.
        These facts are summarized in our main result, Theorems \ref{structure}, to which we also give two variants, 
        Theorems \ref{secondstructure} and \ref{thirdstructure}. Section \ref{productsection} introduce a notion of product of 
        compatible q.o.a.g's. We define an analog of Hahn's product for compatible q.o.a.g's, which we call the compatible product. In Section \ref{Hahntheoremsection}, we prove 
        a generalization of Hahn's embedding theorem for quasi-ordered groups (Theorem \ref{Hahnforqogroups}) which uses our notion of compatible 
        product. Since compatible q.o.a.g's are composed of an ordered part and a valued part, it is natural to ask whether 
        the elementary equivalence of two compatible q.o.a.g's is determined by the elementary equivalence of 
        their respective ordered parts and by the elementary equivalence of their respective valued parts. This is the subject of 
        section \ref{Fefermanvaughtsection}, where we show in particular that 
        the compatible product of an ordered group by a valued group preserves elementary equivalence (Theorem \ref{Fefermanvaughttheorem}).
        This result will be useful 
        for applications to C-minimal groups in Theorem \ref{prodCminimal}.
        In Section \ref{minimalitysection}, we introduce the notion of quasi-order-minimality for compatible quasi-ordered abelian groups.
        We then inquire into a question which Françoise Delon asked us:  
  Are compatible q.o.a.g's also C-groups? If yes, how does quasi-order-minimality relate to 
  C-minimality?
        We answer by showing that any compatible q.o.a.g naturally induces a compatible C-relation on the group, and that 
        quasi-order-minimality is then equivalent to C-minimality (see Proposition \ref{Creldunsqo} and Proposition 
  \ref{Cminimal}). We then show that any compatible q.o.a.g obtained as the product of an o-minimal group by a finite valued group is 
  C-minimal (Theorem \ref{prodCminimal}), which allows us to give an example of a C-minimal group which is neither ordered nor valued.
        
        This connection between 
        C-groups and compatible q.o.a.g's makes the latter a particularly useful class of objects for the study of 
        C- groups, so that the study of compatible q.o.a.g could be the first step toward a classification of 
        C-minimal groups. In particular, it should be emphasized that our work on compatible q.o.a.g's presented here played an essential 
        role in the discovery of the main result of \cite{Lehericy} which gives the structure of an arbitrary C-group.           
        Indeed, seeing compatible q.o.a.g's as an example of C-groups 
        enabled us to gain the right intuition on the structure of arbitrary C-groups. 
        The main result of \cite{Lehericy} 
        (Theorem 3.41),
         as well as the methods used to prove it, were 
        inspired by our work on compatible q.o.a.g's presented here (Note that 
        Theorem 3.41 of \cite{Lehericy} bears some similarities with Theorem \ref{structure} of the current paper).        
        It is interesting to note that            
        our structure Theorem \ref{structure} shows that compatible q.o.a.g's are a particularly simple 
        class of C-groups since the set of o-type elements is an initial segment (one can compare this setting to Theorem 
        3.41 of \cite{Lehericy}, which states that an arbitrary C-group can contain any arbitrary alternation of 
        o-type and v-type parts). In a sense, they are the simplest examples of C-groups whose C-relation 
        does not come from an order nor from a valuation, which is why 
        it is natural to try to understand C-groups by first considering compatible q.o.a.g's.
        Results obtained for compatible q.o.a.g's not only gives us an intuition for C-groups 
        but can also potentially lead to a generalization to arbitrary C-groups. In particular, studying model-theoretic properties 
        of compatible q.o.a.g's can lead to a better understanding of C-minimal groups.

     \section{Preliminaries}\label{prelsection}
     
     \textit{Every group considered in this paper is abelian}.
     By \textbf{ordered abelian group} we mean an abelian group $(G,+)$ equipped with a total order satisfying: 
      \begin{equation}\label{axiomoagr}\tag{OG}
       \forall x,y,z\in G, x\leq y\Rightarrow x+z\leq y+z.
      \end{equation}
   An ordered abelian group is always torsion-free (see \cite{Fuchs}).      
     A \textbf{valuation} on a group $G$ (see \cite{Priess}) is a map $v:G\to \Gamma\cup\{\infty\}$ such that:
	    \begin{enumerate}[(i)]
	    \item $\Gamma$ is a totally ordered set,  and this order is extended to $\Gamma\cup\{\infty\}$ by declaring
	    $\gamma<\infty$ for all $\gamma\in\Gamma$. The ordered set $\Gamma$ is called the \textbf{value chain} of the valued group 
	    $(G,v)$.
	    \item For any $g\in G$, $v(g)=\infty\Leftrightarrow g=0$
	    \item For any $g,h\in G$, $v(g+h)\geq \min(v(g),v(h))$ (ultrametric inequality)
	    \item For any $g\in G$, $v(g)=v(-g)$
	    \end{enumerate}  
	 The \textbf{trivial valuation} on $G$ is the valuation $v$ such that 
	 $v(g)=v(h)$ for any $g,h\in G\backslash\{0\}$.
	If $(G,v)$ is a valued group with value chain $\Gamma$, then $G^{\gamma}$ and $G_{\gamma}$ will respectively denote 
	the subgroups $\{g\in G\mid v(g)\geq\gamma\}$ and $\{g\in G\mid v(g)>\gamma\}$, and $B_{\gamma}$ denotes the quotient group 
	$G^{\gamma}/G_{\gamma}$. The pair $(\Gamma,(B_{\gamma})_{\gamma\in\Gamma})$ is called the \textbf{ skeleton} 
	of the valued group $(G,v)$. A natural example of a group valuation is the archimedean valuation associated to 
	an order. If $(G,\leq)$ is an ordered group, we define the \textbf{archimedean valuation} of $(G,\leq)$ 
	as follows: we say that $v(g)\leq v(h)$ if there are $m,n\in\N$ such that 
	$n\vert h\vert\leq m\vert g\vert$. We say that the ordered group $(G,\leq)$ is \textbf{archimedean} if 
	the archimedean valuation associated to $\leq$ is trivial.	
	A particularly interesting class of valuations are the $\Z$-module valuations considered in 
	\cite{SKuhlmann}. A valuation $v$ on a group $G$ is called a \textbf{$\Z$-module valuation} if $v(ng)=v(g)$ holds for every 
	$n\in\Z\backslash\{0\}$ and every $g\in G$. $\Z$-module valuations appear naturally on the value group of 
	valued fields. The archimedean valuation of an ordered abelian group is a $\Z$-module valuation. 
	If $(G,\psi)$ is the asymptotic couple associated to a H-field (see \cite{Aschdries2}), then $\psi$ is a 
	$\Z$-module valuation.
	
	We recall the notion of Hahn product: if $(B_{\gamma})_{\gamma\in\Gamma}$ is an ordered family of groups, 
	we define the \textbf{Hahn product} of the family $(B_{\gamma})_{\gamma\in\Gamma}$ as the group 
	$\Hahn{B}:=\{(g_{\gamma})_{\gamma\in\Gamma}\in\prod_{\gamma\in\Gamma}B_{\gamma}\mid\supp(g)\text{ is well-ordered}\}$, 
	where $\supp(g)$ denotes the support of $g$.	
	The group $\Hahn{B}$ is naturally endowed with a valuation defined as $v(g):=\min\supp(g)$. If 
	$(B_{\gamma},\leq_{\gamma})_{\gamma\in\Gamma}$ is an ordered family of ordered groups, we define the \textbf{lexicographic product} of the family 
	$(B_{\gamma},\leq_{\gamma})_{\gamma\in\Gamma}$ as the ordered group $(G,\leq)$, where $G=\Hahn{B}$ and $\leq$ is defined as follows: 
	we say that $g=(g_{\gamma})_{\gamma}\leq h=(h_{\gamma})_{\gamma}$ if $g_{\delta}\leq_{\delta}h_{\delta}$ where 
	$\delta=v(g-h)$. 		
	We recall two versions of Hahn's embedding theorem, one for ordered groups and the other one for groups endowed with 
	a $\Z$-module valuation:
	
	\begin{Thm*}[Hahn's embedding theorem for ordered groups, see \cite{Fuchs}]
	Let $(G,\leq)$ be an ordered group. Then $(G,\leq)$ is embeddable into a lexicographic product of 
	archimedean ordered groups.
	\end{Thm*}

	    \begin{Thm*}[Hahn's embedding theorem for $\Z$-module valuations, see \cite{SKuhlmann}]
	     Let $G$ be a divisible group and $v$ a $\Z$-module valuation on $G$ with skeleton 
	     $(\Gamma,(B_{\gamma})_{\gamma\in\Gamma})$. 
	     There is a group embedding $\phi:G\to\Hahn{B}$ and an automorphism of ordered set $\psi:\Gamma\to\Gamma$ such that 
	     $\psi(v(g))=\min\supp(\phi(g))$ for all $g\in G$ (in other words, $\phi$ is an embedding of valued groups).
	    \end{Thm*}

     A \textbf{quasi-order} (q.o) is a binary relation which is reflexive and transitive. If 
     $\precsim$ is a quasi-order on a set $A$
 it induces an equivalence relation on $A$ by $a\sim b$ if and only if $a\precsim b\precsim a$. 
 We say that a q.o $\precsim$ is \textbf{total}
 if for every $a,b\in A$, either $a\precsim b$ or $b\precsim a$ holds.
 \textit{Unless explicitly stated otherwise, every q.o considered in this paper is total}.
 
 \begin{Notation}
  The symbol $\precsim$ will always denote a quasi-order, whereas $\leq$ is exclusively used to denote an order.
 The symbol $\sim$ will always denote the equivalence relation induced by the  quasi-order $\precsim$ and
  $cl(a)$ will denote the class of $a$ for this equivalence relation. The notation
  $a\precnsim b$ means $a\precsim b\wedge a\nsim b$. If $(A,\precsim)$ is a quasi-ordered set, $a\in A$ and $S\subseteq A$ then
  the notation $S\precsim a$ (respectively $S\precnsim a$) means 
  $s\precsim a$ (respectively $s\precnsim a)$ for all $s\in S$.
 \end{Notation}
Note that a quasi-order is an order
 if and only if $cl(a)=\{a\}$ for every $a$. 
 If $(A,\precsim)$ is a quasi-ordered set then $\precsim$ 
 induces an order on the quotient $A/\sim$ by $cl(a)\leq cl(b)$ if and only if $a\precsim b$. Note that
 a q.o
 $\precsim$ is total if and only if it induces a total order on
 $A/\sim$. 
 If $S$ is a subset of $A$, we say that $S$ is
 \textbf{$\precsim$-convex} in $A$ if for any $s,t\in S$, for any $a\in A$, 
 $s\precsim a\precsim t$ implies $a\in S$. 
 We say that $S$ is an \textbf{initial segment} of $A$ if for any $s\in S$ and any 
 $a\in A$, $a\precsim s$ implies $a\in S$.
 If $\precsim$ and $\etoile$ are two q.o's on $A$, we say that $\etoile$ is a \textbf{coarsening} of $\precsim$, or that 
 $\precsim$ is \textbf{finer} than $\etoile$, if $a\precsim b\Rightarrow a\etoile b$ for all $a,b\in A$.

 \begin{Ex}
 Let $(G,v)$ be a valued group.
 Then $a\precsim b\Leftrightarrow v(b)\leq v(a)$ defines a total quasi-order on $G$, called
 the quasi-order induced by $v$. We say that a q.o on a group is \textbf{valuational} if it is induced by a valuation. 
 A q.o induced by a $\Z$-module valuation is called \textbf{$\Z$-module valuational}. Note that the ultrametric inequality 
 can be reformulated in the language of q.o's as $g\precsim h\Rightarrow g+h\precsim h$. We will say that $(G,\precsim)$ is 
 a \textbf{valuationally quasi-ordered group} if $\precsim$ is the q.o induced by a valuation.
 \end{Ex} 
  
 In this paper, a \textbf{quasi-ordered group} is just a group endowed with a quasi-order without any further
 assumption.
 If $(G,\leq)$ is an ordered abelian group and $H$ a convex subgroup of $G$,
 there is a classical notion of the order induced by $\leq$ on the quotient 
 $G/H$ (see \cite{Fuchs}). In  our work with quasi-ordered groups it will also be practical to consider quotients.

\begin{Lem}\label{quotient}
 Let $(G,\precsim)$ be a quasi-ordered group
 and $H$ a subgroup of  $G$ such that the following condition is satisfied:
  \[\forall g_1,g_2\in G((g_1-g_2\notin H\wedge g_1\precsim g_2)\Rightarrow (\forall h_1,h_2\in H, g_1+h_1\precsim g_2+h_2))\]
 Then $\precsim$ induces a total q.o on the quotient $G/H$ defined by:
 \[g+H\precsim h+H\Leftrightarrow g-h\in H\vee (g-h\notin H\wedge g\precsim h)\]
\end{Lem}
\begin{proof}
 The fact that this relation is well-defined follows directly from the assumption.
 This relation is clearly reflexive and total, we just have to check that it is transitive.
   Assume $f+H\precsim g+H\precsim h+H$. 
  If $f-g$ and $g-h$ are both in $H$ then so is $f-h$.   
   If $g-h$ and $f-g$ both lie outside of $H$ then we have $f\precsim g\precsim h$ so $f\precsim h$, so $f+H\precsim h+H$   
   Assume $g-h\notin H$ and  $f-g\in H$. We have $g\precsim h$ and by applying our assumption this implies
   $g-g+f\precsim h+0$ i.e $f\precsim h$, so $f+H\precsim h+H$. The case where $g-h\in H$ and $f-g\notin H$ is similar.
\end{proof}
   
     If $(G,\precsim_G)$ and $(H,\precsim_H)$ are two q.o groups and $\phi:G\to H$ a map, we say that 
     $\phi$ is a \textbf{homomorphism of q.o groups} if it is a homomorphism of groups such that 
     $g_1\precsim_G g_2\Rightarrow\phi(g_1)\precsim_H\phi(g_2)$ for every $g_1,g_2\in G$. Note that 
     a bijective homomorphism of q.o groups is not necessarily an isomorphism of 
     q.o groups, i.e the condition $\phi(g_1)\precsim_H \phi(g_2)\Rightarrow g_1\precsim_G g_2$ is not necessarily satisfied. 
     To see this, consider $\Q^2$ endowed with the lexicographic product $\leq$ of the usual order of $\Q$. Now let 
     $\precsim$ denote the valuational q.o associated to the archimedean valuation of $(\Q^2,\leq)$. The identity map on $\Q^2$ is 
     a homomorphism of q.o groups from $(\Q^2,\leq)$ to $(\Q^2,\precsim)$, but it is not an isomorphism.
     
     Sections \ref{Fefermanvaughtsection} and \ref{minimalitysection} deal with model-theoretic aspects of compatible 
     q.o.a.g's. The natural language for compatible q.o.a.g's is $\langue:=\{0,+,-,\precsim\}$, where 
     $-$ is interpreted as a unary relation and $\precsim$ as the q.o. We will use 
     $\{x_1,x_2,\dots\}$ and $\{y_1,y_2,\dots\}$ as sets of variables. 
     We denote tuples with a bar, thus $\bar{x}$ means a tuple $(x_1,\dots,x_n)$ for a certain $n\in\N$ and 
     $\bar{g}$ means $(g_1,\dots,g_n)$ for a certain $n\in\N$. If $\bar{g}$ and $\bar{h}$ are tuples of elements of a group $G$ both of
     the same length $n$, 
     we denote by 
     $\bar{g}+\bar{h}$  the tuple $(g_1+h_1,\dots,g_n+h_n)$ and if $H$ is a subgroup of $G$ we denote by 
     $\bar{g}+H$ the tuple $(g_1+H,\dots,g_n+H)$ of elements of $G/H$. Equality between formulas is denoted by 
     $\equiv$ to avoid confusion with the equality symbol of the language, which is $=$.

     \section{Compatible quasi-orders}\label{qoagsection}
     
     Our goal is to find a good generalization of orders and valuations on abelian groups. To this end, we get inspiration
     from Fakhruddin's work and introduce the following definition:
     
     \begin{Def}
      Let $G$ be an abelian group and $\precsim$ a q.o on $G$. We say that $\precsim$ is \textbf{compatible (with $+$)}
       if it satisfies the axioms:\begin{itemize}
                                          \item[$(Q_1)$] $\forall x(x\sim0\Rightarrow x=0)$
	      \item[$(Q_2)$] $\forall x,y,z(x\precsim y\nsim z\Rightarrow x+z\precsim y+z)$
                                         \end{itemize}

      We also say that the pair $(G,\precsim)$ is a \textbf{compatible q.o.a.g }(quasi-ordered abelian group).
     \end{Def}

     As in the case of fields, it is easy to check that if $(G,\precsim)$ is actually an ordered abelian group or 
     if $\precsim$ is a valuational q.o then $\precsim$ is compatible with $+$. However, we have no analog of Fakhruddin's dichotomy, i.e 
     there are some compatible q.o's which are not an order and do not come from a valuation. 
     We will 
     show this now by giving three different examples where the q.o is neither an order nor valuational. 
     One could directly check that these q.o's satisfy axioms $(Q_1)$ and $(Q_2)$, but this will actually be a  consequence 
     of Theorem \ref{structure}.

      \begin{Ex}\label{contreexemple}
      \begin{enumerate}[(a)]
       \item Consider the group
$G:=\entiersrel^2$ endowed with the following quasi-order:\newline
$(a,b)\precsim (c,d)\Leftrightarrow (c\neq0)\vee(c=a=0\wedge b\leq d)$, where $\leq$ is the usual order of
$\entiersrel$.
The q.o is an order on $G^o:=0\times\entiersrel$ (it coincides with $\leq$) so it cannot be valuational.
However, it cannot be an order on $\entiersrel^2$ since  we have $(a,b)\sim (c,d)$ for any 
$a,c\neq0$ and any $b,d$.
	\item Set $G:=\Z$ and $G^o:=5\Z$. Endow $G^o$ with its usual order $\leq$, and extend $\leq$ to a q.o $\precsim$ on 
     $G$ by declaring that $f\precnsim g\sim h$ for any $f\in G^o$ and $g,h\notin G^o$. Then $(G,+,\precsim)$ is a compatible 
     q.o.a.g.
	\item Let $(K,\leq,\sigma)$ be an ordered difference field with the assumptions of Section 5 of 
	\cite{KuhlmannPointMatusinski}. In \cite{KuhlmannPointMatusinski}, the authors defined 
an equivalence relation $\sim_{\sigma}$ on $P_K:=K^{\geq0}\backslash K_v$, where $K_v$ is the valuation ring of $v$. 
This equivalence relation is related to the difference rank of $(K,\leq,\sigma)$ (see Theorem 5.3 of \cite{KuhlmannPointMatusinski}).
They also showed that the $\sim_{\sigma}$-classes are naturally ordered. This gives rise to a q.o on 
$P_K$ defined as $a\precsim_{\sigma}b\Leftrightarrow cl_{\sigma}(a)\leq cl_{\sigma}(b)$ ($cl_{\sigma}$ denotes the $\sim_{\sigma}$-class of $a$). This q.o can easily be extended to 
$K\backslash K_v$ by declaring that $-a\sim_{\sigma}a$ for every $a$. Note that $\precsim_{\sigma}$ satisfies the ultrametric 
inequality on $K\backslash K_v$. Now define a q.o $\precsim$
on $K$ as follows: if $a,b\in K_v$, then $a\precsim b\Leftrightarrow a\leq b$; if 
$a,b\notin K_v$ then $a\precsim b\Leftrightarrow a\precsim_{\sigma}b$; finally, declare 
$a\precnsim b$ whenever $a\in K_v$ and $b\notin K_v$. This makes $(K,+,\precsim)$ a compatible q.o.a.g. 
The q.o $\precsim$ contains both the information about the order $\leq$ of $K$ and some information about the 
$\sigma$-rank of $K$.
Note that we can do a similar construction with $H$-fields if we replace $\precsim_{\sigma}$ by the q.o 
$\precsim_{\phi}$ defined in Section 3.2 of \cite{KuhlmannLehericy}.
      \end{enumerate}

\end{Ex}

\begin{Rem}\label{axiomq2}
\begin{enumerate}
 \item In the case where $\precsim$ is actually an order, note that 
 $(Q_2)$ is technically weaker than (\ref{axiomoagr}) because of the condition ``$y\nsim z$''.
 However, the only ordered group which satisfies $(Q_2)$ but not (\ref{axiomoagr}) is 
 $\Z/2\Z$ with the order $0<1$ (see Proposition \ref{oagr}), so 
 $(Q_2)$ and (\ref{axiomoagr}) are essentially equivalent for orders.
 \item The condition ``$y\nsim z$'' in $(Q_2)$ is essential if we want to include valuational q.o's.
	Indeed, if $\precsim$ is a valuational q.o, and if we take
			$x\neq y=-z$ such that $x\precsim y$, we then have $x+z\neq0$ and $y+z=0$ which implies
			$y+z\precnsim x+z$.
 \end{enumerate}
\end{Rem}

 \subsection{o-type and v-type elements}

  We now fix a compatible q.o.a.g $(G,\precsim)$ and investigate its structure. As mentioned in the introduction, 
  we want to show that $(G,\precsim)$ 
  is a mix of ordered and valued groups, which is why we need to distinguish two kinds of elements in $G$.
    We say that $g\in G$ is \textbf{o-type} if $cl(g)=\{g\}\wedge ord(g)\neq2$ and we
    say that it is \textbf{v-type} if $\{g\}\varsubsetneq cl(g)\vee 2g=0$. Note that $0$ is both o-type and v-type and that it is the only element 
    of $G$ with this property.
We set  $G^o:=\{g\in G\mid g\text{ is o-type }\}$ and $G^v:=G\backslash G^o$. Note that $\precsim$ is an order on 
$G^o$. The condition $ord(g)\neq 2$ in the definition of o-type is motivated by Example
\ref{z2} below: we want $\precsim$ to satisfy (\ref{axiomoagr}) on $G^o$ so we do not want to count 
$1$ in $Z/2Z$ as an o-type element. The following proposition gives a characterization of o-type and v-type elements:
   
   \begin{Prop}\label{otypevtype}
 
 For any $g\in G$, $g$ is v-type if and only if $g\sim -g$ if and only if $0\precsim g\wedge 0\precsim-g$.
 Equivalently: $g$ is o-type if and only if $g\nsim -g\vee g=0$ if and only if $g\precsim 0\vee -g\precsim 0$
\end{Prop}

   \begin{proof}
    Since the first line is the contra-position of the second we just need to prove one of them.
    If $g$ is o-type with $g\neq0$, then by definition of o-type we have $cl(g)=\{g\}$ and $g\neq -g$ so 
    $-g\notin cl(g)$ i.e $g\nsim -g$.     
    If $g\nsim -g$, then by $(Q_2)$ the inequality $0\precsim g$ implies $-g\precsim g-g=0$. This shows that 
    ``$g$ is o-type'' $\Rightarrow (g\nsim -g\vee g=0)\Rightarrow (g\precsim 0\vee -g\precsim 0)$. Now assume 
    $g\precsim 0\vee -g\precsim 0$ holds and let us show that $g$ is o-type. Without loss of generality we may assume 
    $g\precsim 0$ and $g\neq0$. By $(Q_1)$, we have $g\precsim 0\nsim -g$, which by 
     $(Q_2)$ implies $0\precsim -g$. If $g\sim -g$ were true, we would then have 
    $g\sim 0$, which is a contradiction to $(Q_1)$. Thus, $g\nsim -g$, which in particular implies $ord(g)\neq2$.
    Let $h\in G$ with $h\sim g$; we have $h\precsim g$ and $g\precsim h$. Since $-g\nsim g\sim h$, we can apply 
    $(Q_2)$ to both inequalities and we get $h-g\precsim 0$ and $0\precsim h-g$ which implies $g-h\sim0$, which
    by $(Q_1)$ means $h=g$. This proves
    $cl(g)=\{g\}$, so $g$ is o-type.
   \end{proof}

    As mentioned in Remark \ref{axiomq2}, $(Q_2)$ is not the same as axiom (\ref{axiomoagr}) of 
    ordered abelian groups, and it can in fact happen
    that a compatible quasi-order is an order but does not satisfy (\ref{axiomoagr}):
    \begin{Ex}\label{z2}
  If we order $\entiersrel/2\entiersrel$ by $0<1$ then $(\entiersrel/2\entiersrel,\leq)$ does not satisfy
  (\ref{axiomoagr}) but it is a compatible q.o.a.g. More precisely, $\leq$ is the q.o induced by the 
   trivial valuation on $G$.
  \end{Ex}
  
    Remarkably, this is the only pathological case. To show this we need the following lemma:
    
     \begin{Lem}
  Assume $\precsim$ is an order and assume that $G$ has an element of order $2$. Then $G=\entiersrel/2\entiersrel$.
 \end{Lem}

 \begin{proof}
  Let $g$ be an element of order $2$. Then $g$ is  v-type, which by Proposition \ref{otypevtype} implies 
  $0\precsim g$.
  Let $h\neq g$; since $\precsim$ is an order we have $h\nsim g$, so we can apply $(Q_2)$ to $0\precsim g$ which yields
  $h\precsim g+h$. If $h\neq 0$ then  $g\nsim g+h$ so we can apply $(Q_2)$ to the previous inequality and get
  $g+h\precsim g+g+h=h$, hence $h\sim g+h$, but since $\precsim$ is an order this implies $h=g+h$ hence $g=0$, which is a contradiction. This
  proves that $h\neq g$ implies $h=0$.
  We thus have $G=\{0,g\}\cong\entiersrel/2\entiersrel$
 \end{proof}

 \begin{Prop}\label{oagr}
	 Let $(G,\precsim)$ be a compatible q.o.a.g.	 
	 If $\precsim$ is an order and if $G\neq \entiersrel/2\entiersrel$, then $(G,\precsim)$ is an ordered abelian group,
	 i.e (\ref{axiomoagr}) is satisfied.
	\end{Prop}
 
\begin{proof}
 
  We want to prove:
  $\forall x,y,z\in G,x\precsim y\Rightarrow x+z\precsim y+z$.
  Since $\precsim$ is an order, we have $y\sim z\Rightarrow y=z$ for any $y,z\in G$. Thus, we only have to 
  consider the case where $y=z$, since the other cases are given by axiom $(Q_2)$.
  Assume then that $x\precsim y$. Since $G\neq\entiersrel/2\entiersrel$, the previous lemma ensures that 
  $y\neq -y$, so $y\nsim -y$. We can then apply $(Q_2)$ to $x\precsim y$ and we get 
  $x-y\precsim 0$. Since $2y\neq0$, we can again apply $(Q_2)$ to this inequality and obtain $x+y\precsim y+y$, which is what 
  we wanted.
 \end{proof}
 
 \begin{Rem}
  Since the case $\Z/2\Z$ is somewhat degenerate, it would be tempting to exclude this case from the definition 
  of compatible q.o.a.g's. However, this seems rather unreasonable in view of Proposition 
  \ref{convexsubgroupProp} below. Indeed, we want the class of compatible q.o.a.g's to be stable under quotient by 
  convex subgroups, which would not be the case if $\Z/2\Z$ were excluded. 
 \end{Rem}

 An immediate consequence of Proposition \ref{oagr} is the following:
 
 \begin{Prop}\label{qoagotype}
   The compatible q.o.a.g $(G,\precsim)$ is an ordered abelian group  if and only if every element of $G$ is o-type.
   \end{Prop}
 
  We are now going to investigate $G^o$ and $G^v$ in more details and show that they have remarkable properties.
   
   \subsection{Properties of $G^o$ and $G^v$}
   Set $\Gamma=G/\sim$ and denote by $\leq$ the order induced by $\precsim$ on $\Gamma$. For any 
   $\gamma\in \Gamma$, set \newline
   $G^{\gamma}:=\{g\in G\mid cl(g)\leq \gamma\}$ and $G_{\gamma}:=\{g\in G\mid cl(g)<\gamma\}$
  
   \begin{Rem}
    If  $\precsim$ is the q.o induced by a valuation $v$ then 
    $\Gamma$ with the reverse ordering of $\leq$ is isomorphic to $v(G)$. 
    In that case, our definition of $G^{\gamma},G_{\gamma}$ coincides 
    with the definition given in Section \ref{prelsection} for valued groups, i.e 
    $G^{\gamma}=\{g\in G\mid v(g)\geq\gamma\}$ and $G_{\gamma}=\{g\in G\mid v(g)>\gamma\}$.
    If $(G,\leq)$ is an ordered abelian group then $\Gamma$ and $G$ are isomorphic as ordered sets.
   \end{Rem}

   The following two lemmas will have important consequence on $G^o$ and $G^v$:
   
   \begin{Lem}\label{lemma1}
 If $h$ is v-type and $g\nsim h$ then $g-h\sim g+h$.
\end{Lem}

\begin{proof}
 Since $h$ is v-type we have by Proposition \ref{otypevtype} $h\sim -h$,  which means $h\precsim -h\precsim h$. Since $g\nsim h$, we can 
 apply $(Q_2)$ to these inequalities and we get $g+h\precsim g-h\precsim g+h$.
\end{proof}
   
   \begin{Lem}\label{sousgroupes}
Let $H$ be an initial segment of $G$ containing $G^o$. Then $H$ is a subgroup of $G$.
\end{Lem}

\begin{proof} 
 Since $0$ is o-type, $0\in G^o\subseteq H$.
 Let $h\in H$. If $-h\sim h$, then in particular $-h \precsim h$, so $-h\in H$ because $H$ is an initial segment of $G$. If
 $-h\nsim h$ then by Proposition \ref{otypevtype} $h$ is o-type so $-h\in G^o\subseteq H$. This shows that 
 $H$ is closed under taking the inverse, and we are now going to show that it is closed under addition.
 Let $g,h\in H$; we can assume $h\neq 0$. If $g\precsim 0$ then $(Q_2)$ implies
 $g+h\precsim h\in H$  hence
 $g+h\in H$ because $H$ is an initial segment of $G$. 
 Now assume $0\precsim g$. If $-(g+h)\notin H$, then in particular $g\nsim -(g+h)$ (because $H$ is an initial segment of $G$), so we can apply $(Q_2)$ to the
 inequality $0\precsim g$ and get $-(g+h)\precsim -h\in H$, so $-(g+h)\in H$, which is a contradiction. 
 Thus, $-(g+h)\in H$, and since $H$ is closed under taking the inverse this implies $g+h\in H$. 
\end{proof}
   
   We can now give the main properties of $G^o$:
   \begin{Prop}\label{segin}
    $G^o$ is an initial segment and a subgroup of $G$. In particular, $(G^o,\precsim)$ is an ordered abelian group.
   \end{Prop}

   \begin{proof}
    Let $h\in G^v$ and $g\in G^o$. We want to show that $g\precnsim h$.
    If $h=g$ then $h$ would be v-type and o-type, which can only happen if $h=0$ which is
    excluded by definition of $G^v$, so $h\neq g$. Since $g$ is o-type this implies $g\nsim h$.
    By Lemma \ref{lemma1}, we then have $g+h\sim g-h$.
    Since $h$ is v-type then by Lemma \ref{otypevtype} $-h$ is also v-type so $-h\neq g$ and $0\precsim -h$, which by $(Q_2)$ implies $g\precsim g-h$.
    Assume $h\nsim g-h$. We can apply $(Q_2)$ to the previous inequality and get
    $g+h\precsim g$, so we have $g+h\precsim g\precsim g-h\sim g+h$, which 
 means $g\sim g+h$, which  contradicts the fact that $g$ is o-type.
 Thus, we have $g\precsim g-h\sim h$, which is what we wanted.  
 This shows that $G^o$ is an initial segment, and by Lemma \ref{sousgroupes} it follows that $G^o$ is a subgroup of 
 $G$. By Proposition \ref{qoagotype} $(G^o,\precsim)$ is then an ordered abelian group.
   \end{proof}

 The main property of $G^v$ is given by the ultrametric inequality satisfied by v-type elements:

\begin{Prop}[ultrametric inequality for v-type elements]\label{ultrametric}
 $G^v$ is a final segment of $G$. Moreover, for any $g\in G^v$ and $h\in G$ we have \newline
 $cl(g+h)\leq\max(cl(g),cl(h))$.
 If $h\precnsim g$ then $g\sim g+h$.
\end{Prop}

\begin{proof}
  The fact that $G^v$ is a final segment follows directly from Proposition \ref{segin}. Take $g\in G^v$ and 
  $h\in G$. We can assume that $h\precsim g$: otherwise we have $h\in G^v$ so we can exchange the roles of $g$ and $h$.
  By Proposition \ref{segin} $G^{cl(g)},G_{cl(g)}$ contain $G^o$. By Lemma 
  \ref{sousgroupes} it follows that they are subgroups of $G$. In particular, since $g,h\in G^{cl(g)}$ we have 
  $g+h\in G^{cl(g)}$, hence $g+h\precsim g$. If $h\precnsim g$ we even have $h\in G_{cl(g)}$ but 
  $g\in G^{cl(g)}\backslash G_{cl(g)}$ hence $g+h\in G^{cl(g)}\backslash G_{cl(g)}$ which means $g\sim g+h$.
 \end{proof}
 
 We can reformulate Proposition \ref{ultrametric} by saying that $\precsim$ behaves like a valuation on $G^v$:
 
 \begin{Prop}\label{valuation}
  Set $\Gamma^v:=cl(G^v)$ and take $\gamma_0\in\Gamma$ with $\gamma_0<\Gamma^v$. Let $\leq^{\ast}$ be the reverse 
  order of $\leq$ on $\Gamma_v\cup\{\gamma_0\}$. Define $v$ on $G$ by:
  \[v(g)=\left\{\begin{array}{cc}
                 cl(g) & \text{ if } g\in G^v\\
                 \gamma_0 & \text{ if } 0\neq g\in G^o\\
                 \infty & \text{ if } g=0
                \end{array}\right.\]
 Then $v:G\to (\Gamma_v\cup\{\gamma_0,\infty\},\leq^{\ast})$ is a valuation and we have 
 $g\precsim h\Leftrightarrow v(g)\geq v(h)$ for any $g,h\in G^v$.
 \end{Prop}
\begin{proof}
 It suffices to show that $v(g+h)\geq \min(v(g),v(h))$ for any $g,h\in G$. If $g$ or $h$ is in $G^v$ this is 
 given by Proposition \ref{ultrametric}. If $g,h\in G^o$ this is given by the fact that $G^o$ is a subgroup of $G$.
\end{proof}

 As a special case of Proposition \ref{valuation} we have a v-type analog of Proposition \ref{qoagotype}:
 \begin{Prop}
  The  compatible q.o $\precsim$ is valuational if and only if every element of $G$ is v-type.
  In that case, the map 
  $cl:G\to \Gamma$ with reverse order on $\Gamma$ and with $\infty:=cl(0)$ is
  a valuation, and $\precsim$ is the q.o induced by 
  this valuation.
 \end{Prop}

 \subsection{Quasi-order induced on a quotient}
 
 It is known that if $(G,\leq)$ is an ordered abelian group and if $H$ is a convex subgroup, then $\leq$ 
 naturally induces an order on the quotient $G/H$ (see \cite{Fuchs}). We now show that the same is true in the case of compatible
 q.o.a.g's, which will allow us to give a more elegant formulation of Proposition \ref{valuation}.
 We start by describing convex subgroups:
 
 \begin{Prop}\label{convex}
  Let $H$ be a convex subgroup of $G$. Then either $H\subseteq G^o$ or 
  $G^o\subseteq H$. If the latter holds, then $H$ is an initial segment.
 \end{Prop}

 \begin{proof}
  Assume $G^o$ is not contained in $H$, so there exists some o-type element $g$ with $g\notin H$. 
  Without loss of generality, we can assume $0\precsim g$.
   By convexity of $H$, we have $H\precnsim g$, which by Proposition \ref{segin} implies that every element of $H$ is
   o-type.
   Now assume $G^o\subseteq H$ and take $h\in H$ and $g\in G$ with $g\precsim h$. 
   If $g\in G^o$ then $g\in H$ by assumption; if $g\notin G^o$, then by Lemma \ref{otypevtype} we have 
   $0\precsim g\precsim h$ hence $g\in H$ by convexity of $H$.
   This shows that $H$ is an initial segment.
 \end{proof}

  We will need the following lemma to define a q.o on quotients:
 \begin{Lem}\label{lemmeconvex}
  Let $H$ be a convex subgroup of $G$ and take $g\in G\backslash H$. Then we either have 
  $g\precnsim H$ or $H\precnsim g$; moreover,
  $g\precnsim H$ if and only if $g+h\precnsim H$ for all $h\in H$ (equivalently:
  $H\precnsim g$ if and only if $H\precnsim g+h$ for every $h\in H$).
 \end{Lem}

 \begin{proof}
  The fact that we have $g\precnsim H\vee H\precnsim g$ is a direct consequence of the convexity of $H$.
  If $g\precnsim H$ then in particular $g\precnsim 0$ which by $(Q_2)$ implies $g+h\precsim h$ and since $g+h\notin H$ this implies 
  $g+h\precnsim H$.
 \end{proof}

 \begin{Prop}\label{convexsubgroupProp}
  Let $H$ be a convex subgroup of $G$. Then $H$ satisfies the condition of Lemma \ref{quotient}. Moreover, 
  the q.o induced on $G/H$ is again compatible, and the canonical projection from $G$ to $G/H$ is a homomorphism of q.o groups.
 \end{Prop}

 \begin{proof}
  
   We first prove that the condition of Lemma \ref{quotient} is satisfied.
   Let $g_1,g_2\in G$ with 
  $g_1-g_2\notin H$ and $g_1\precsim g_2$ and let $h_1,h_2\in H$. We want to show that $g_1+h_1\precsim g_2+h_2$.
  If $g_1\in H$ then $g_2\notin H$ and by convexity of $H$ we have $H\precnsim g_2$ which by 
  lemma \ref{lemmeconvex} implies $g_1+h_1\precnsim g_2+h_2$. The case where $g_2\in H$ is similar,
  so we can assume that 
  $g_1,g_2$ are not in $H$.  We first consider the case $g_1-g_2\precsim 0$. If $g_1-g_2\precsim 0$, then $g_1-g_2\precnsim H$ so 
  by lemma \ref{lemmeconvex} we have $g_1-g_2+h_1\precsim h_2$, and since $g_2\notin H$ we have $g_2\nsim h_2$ so
  $(Q_2)$ implies $g_1+h_1\precsim g_2+h_2$. 
  We now consider the case $0\precnsim g_1-g_2$. In this case,  $g_2$ must be v-type (otherwise we would have a contradiction 
  with $g_1\precsim g_2$),
  which by Proposition \ref{ultrametric} implies $g_2\sim g_2+h_2\in G^v$. 
  If $g_1$ is o-type then by Proposition 
  \ref{convex} $h_1$ is also o-type so $g_1+h_1\in G^o$ hence $g_1+h_1\precsim g_2+h_2$ by Proposition \ref{segin}; if 
  $g_1$ is v-type then $g_1+h_1\sim g_1$ by Proposition \ref{ultrametric} hence $g_1+h_1\precsim g_1+h_2$.

  This proves that the condition of Lemma \ref{quotient} is satisfied, and it then follows directly from 
  Lemma \ref{quotient} that $\precsim$ induces a q.o on $G/H$ via the formula also given in Lemma \ref{quotient}.  
  It is clear from the definition of the induced q.o on $G/H$ that the canonical projection is a homomorphism of 
  q.o groups.
  Now let us prove that this q.o is compatible with $+$.
   It is clear from the definition of the induced q.o that, if $g\notin H$, then   
   $g+H\nsim 0+H$, which shows that $(G/H,\precsim)$ satisfies $(Q_1)$.
  Now assume that $g+H\precsim h+H\nsim f+H$. If $g-h\notin H$ then $g\precsim h$. Since $h+H\nsim f+H$ we have 
  $h\nsim f$, hence $g+f\precsim h+f$, hence $g+f+H\precsim h+f+H$. If $g-h\in H$ then $g+f-(h+f)\in H$ hence
  $g+f+H\precsim h+f+H$. This proves $(Q_2)$. 
 \end{proof}
 \begin{Rem}\label{remarkonquotientqo}\label{sameHclassareequivalent}
   \begin{enumerate}[(1)]
    \item If $(G,\precsim)$ is an ordered abelian group then this is the definition of the order induced on $G/H$
  (see \cite{Fuchs}), which is why it is natural to define the q.o on the quotient by the formula given in Lemma \ref{quotient}.
    \item If $H$ has a group complement $F$ in $G$, then $F$ is canonically isomorphic to $G/H$, and it is then easy to see that 
    the q.o induced by $\precsim$ on the quotient $G/H$ coincides with the restriction of $\precsim$ to $F$.
    \item If $G^o\subseteq H$ and $g\notin H$, then for any $h\in G$, 
    $g-h\in H$ implies $g\sim h$ (indeed, we have $g\in G^v$, and by Proposition \ref{convex} we have $h-g\precnsim g$. It then follows from 
    \ref{ultrametric} that $g\sim g+h-g=h$)
    \item  It is clear from the definition of the induced q.o that 
    $g\precsim h\Rightarrow g+H\precsim h+H$ and $g+H\precnsim h+H\Rightarrow g\precnsim h$ hold for all $g,h\in G$. It also follows from the previous remark that
    $g\precsim h\Leftrightarrow g+H\precsim h+H$ and $g+H\precnsim h+H\Leftrightarrow g\precnsim h$
    is true when $G^o\subseteq H$ and $h\notin H$. This remark will be useful for later proofs.
   \end{enumerate}

 \end{Rem}

 As we noted in Section \ref{prelsection}, a bijective homomorphism is in general not an isomorphism. A consequence 
 of this is that there is no equivalent of the fundamental homomorphism theorem of groups, i.e a 
 homomorphism of q.o groups is not always the product of a projection by an embedding. However, we can say the 
 following:
 
 \begin{Prop}
  Let $(G,\precsim_G)$ and $(H,\precsim_H)$ be two compatible q.o groups,
  $\phi:G\to H$ a homomorphism of q.o groups and $\pi$ the canonical projection from $G$ to $G/\ker\phi$. Then $\ker\phi$ is convex in $G$, so
  $\precsim_G$ induces a compatible q.o
  on $G/\ker\phi$. 
  Moreover, the map $G/\ker\phi\to H, g+\ker\phi\mapsto \phi(g)$ is an injective homomorphism of 
  q.o groups. If moreover $\phi$ satisfies $\phi(g)\precsim_H\phi(h)\Rightarrow g\precsim_Gh$ for every $g,h\in G$, then 
  $g+\ker\phi\mapsto\phi(g)$ is an embedding of q.o groups.
 \end{Prop}
 
 \begin{proof} 
   Let $g,f\in \ker\phi$ and $h\in G$ with $f\precsim_G h\precsim_G g$. Then $\phi(f)=0\precsim_H\phi(h)\precsim_H \phi(g)=0$. By 
   $(Q_1)$, it follows that $h\in\ker\phi$. This proves that $\ker\phi$ is convex, and so by Proposition \ref{convexsubgroupProp}
   $\precsim_G$ induces a q.o on $G/\ker\phi$ via the formula given in Lemma \ref{quotient}.  
   We know from general group theory that the map given by the formula $\psi(g+\ker\phi):=\phi(g)$ is a well-defined 
   injective group homomorphism from $G/\ker\phi$ to $H$. Now let $g,h\in G$ such that 
   $g+\ker\phi\precsim_G h+\ker\phi$. If $g-h\in \ker\phi$ then $\psi(g+H)=\psi(h+H)$. If 
   $g-h\notin H$ then $g\precsim_G h$, and since $\phi$ preserves the q.o it follows that 
   $\psi(g+H)\precsim_H\psi(h+H)$. In any case, we have $\psi(g+H)\precsim_H\psi(h+H)$. 
    Now assume that $\phi$ satisfies $\phi(g)\precsim_H\phi(h)\Rightarrow g\precsim_Gh$ and assume that 
   $\psi(g+H)\precsim\psi(h+H)$ holds. Then we have $g\precsim_Gh$, which by definition of the induced q.o on the quotient 
   implies $g+\ker\phi\precsim_Gh+\ker\phi$.
 \end{proof}

 Proposition \ref{convexsubgroupProp} allows us to reformulate Proposition \ref{valuation}:
 
 \begin{Prop}\label{quotientGo}
  Let $H$ be a convex subgroup of $G$. The induced q.o on $G/H$ is valuational if and only if 
  $G^o\subseteq H$. In particular, $G^o$ is the smallest convex subgroup of $G$ such that  the induced q.o on 
  $G/G^o$ is valuational.
 \end{Prop}
\begin{proof}
 If $G^o\subseteq H$, then it follows from proposition \ref{ultrametric} that $\precsim$ is valuational on $G/H$.
 If $H\varsubsetneq G^o$ then there is $g\in G^o\backslash H$ and we can easily see that $g+H$ is o-type in $G/H$ so
 $G/H$ cannot be valuational.
\end{proof}

  Propositions \ref{segin} and \ref{quotientGo} show that $(G,\precsim)$ is an extension of a valued group by an 
  ordered group. We now define the \textbf{ordered part} of $(G,\precsim)$ as the ordered group 
  $(G^o,\precsim)$ and the \textbf{valued part} of $(G,\precsim)$ as the valued group 
  $(G^o/G,v)$, where $v$ is the valuation corresponding to the q.o induced by $\precsim$ on 
  $G^o/G^v$. We will now express $\precsim$ with a formula in which the order of its ordered part and the 
  valuation of its valued part explicitly appear.

   \begin{Prop}\label{formulafortheqo}
    If $\leq_o$ denotes the restriction of $\precsim$ to $G^o$ and $v$ the valuation corresponding to the q.o induced by 
    $\precsim$ on $G/G^o$, then $\precsim$ is given by the following formula for all $g,h\in G$:\newline
    $g\precsim h\Leftrightarrow (g,h\in G^o\wedge g\leq_o h)\vee(h\in G^v\wedge v(g+G^o)\geq v(h+G^o))$
   \end{Prop}

   \begin{proof}
    Denote by $\etoile$ the q.o given by the formula 
    $g\etoile h \Leftrightarrow (g,h\in G^o\wedge g\leq_o h)\vee(h\in G^v\wedge v(g+G^o)\geq v(h+G^o))$. We show that $\etoile$ coincides 
    with $\precsim$.
    Assume $g\precsim h$. Then $g+G^o\precsim h+G^o$ holds by definition of the induced q.o on $G/G^o$. If $h\in G^v$, it
    directly follows from the definition of $\etoile$ that $g\etoile h$. If $h\notin G^v$, then 
    by \ref{segin} we must have $g\in G^o$, and by definition of $\leq_o$, $g\precsim h$ then implies $g\leq_oh$, which by definition of 
    $\etoile$ implies  
    $g\etoile h$.
    Conversely, assume $g\etoile h$. If $g,h\in G^o$ then $g\leq_o h$ which by definition of $\leq_o$ implies 
    $g\precsim h$. Assume then that $(h\in G^v\wedge v(g+G^o)\geq v(h+G^o))$ holds. 
    If $g-h\in G^o$, then Remark \ref{sameHclassareequivalent}(3) implies $g\sim h$ so $g\precsim h$. Otherwise, we have
     $g\precsim h$ by definition of the q.o induced on $G/G^o$. 
   \end{proof}

 \subsection{Structure theorems}
 We can summarize previous results into the following theorem which gives the structure of a compatible 
 q.o.a.g:
 
 \begin{Thm}[structure theorem]\label{structure}
  Let $(G,\precsim)$ be a q.o.a.g.
  Then $\precsim$ is compatible with $+$ if and only if $G$ admits a subgroup $H$ satisfying the following properties:
 \begin{enumerate}[(1)]
     \item $H$ is an initial segment of $G$.
   \item $(H,\precsim)$ is an ordered abelian group.
   \item There exists a valuation $v$ on $G$ such that $v(H)>v(G\backslash H)$ and 
   $g\precsim h\Leftrightarrow v(g)\geq v(h)$ for every $g,h$ in  $G\backslash H$.
   \end{enumerate}

 \end{Thm}

 \begin{proof}
  We have already showed that if $\precsim$ is compatible with $+$ then $G^o$ satisfies (1)+(2)+(3) (Propositions 
  \ref{segin} and \ref{valuation}).
  Assume there exists $H$ satisfying (1)+(2)+(3).
  $(Q_1)$ is clearly satisfied, so let us prove 
  $(Q_2)$.
  Let $x\precsim y\nsim z$. 
  Assume $y\in H$. Since $H$ is an initial segment this implies $x\in H$. If $z\in H$, then since $H$ is an ordered abelian group
  we have $x\precsim y\Rightarrow x+z\precsim y+z$. If $z\notin H$, then $v(x),v(y)>v(z)$ so 
  $v(z)=v(x+z)=v(y+z)$, and since $v$ and $\precsim$ coincide outside of $H$ this means 
  $z\sim x+z\sim y+z$ so in particular $x+z\precsim y+z$.
  Assume $y\notin H$. Then $x\precsim y$ implies $v(x)\geq v(y)$ and
  $z\nsim y$ implies  $v(z)\neq v(y)$. It follows that $v(x+z)\geq v(y+z)$ and $y+z\notin H$ (otherwise we would have
  $v(y+z)>v(z)$ which would imply $v(y)=v(z)$), hence $x+z\precsim y+z$. 
 \end{proof}
 
\begin{Rem}\label{remarquestructuretheorem}
 
 As we have seen in Proposition \ref{quotientGo}, (1)+(2)+(3) implies:
 \begin{itemize}
 \item[(3$'$)] The q.o $\precsim$ induces a valuational q.o on $G/H$ via the formula:
   \[g+H\precsim h+H\Leftrightarrow g-h\in H\vee(g-h\notin H\wedge g\precsim h)\]
   \end{itemize}
   
 It is tempting to replace (3) by (3$'$) in Theorem \ref{structure}, as (3$'$) seems to be a more elegant reformulation of 
 (3).
 However,
 condition (3) is in general stronger than (3$'$), so that Theorem \ref{structure} becomes false if we replace (3) by (3$'$).
 We can construct an example of a group satisfying 
 (1)+(2)+(3$'$) but which is not compatible:
 Take $G:=(\Z/2\Z)\times\Z$ with the following q.o:
 $(a,b)\precsim (c,d)\Leftrightarrow (a=c\wedge b\leq d)\vee(a<c)$, where $\leq$ is the usual order of $\Z$.
 We have:
 $(0\times\Z,\leq)\precnsim \dots (1,-n)\precnsim\dots\precnsim(1,-1)
 \precnsim(1,0)\precnsim\dots\precnsim(1,n)\precnsim
 \dots$.
 Setting $H=0\times\Z$, $H$ satisfies (1)+(2). It also satisfies (3$'$):
 If $(a,b),(c,d)$ satisfy $(a,b)-(c,d)\notin H$, that means $a\neq c$. From $(a,b)\precsim(c,d)$ follows 
 $a=0\wedge c=1$, and if we take $(0,e_1),(0,e_2)\in H$ then we still have  $(a,b)+(0,e_1)\precsim (c,d)+(0,e_2)$. 
 By Lemma \ref{quotient} this proves
 that the q.o on the quotient is well-defined. Moreover, the induced q.o on the quotient
 $G/H=\Z/2\Z$ is $0\precnsim 1$ which is valuational, 
 so (3$'$) is satisfied. However, $\precsim$ cannot be compatible: the set of o-type elements is $G\backslash\{(1,0)\}$ which is clearly not a 
 group, thus contradicting Proposition \ref{segin}. We can also give an explicit example of axiom $(Q_2)$ failing: take 
 $x:=(0,0), y:=(1,0)$ and $z:=(1,1)$. We have $x\precsim y\nsim z$ but 
 $y+z=(0,1)\precnsim x+z=(1,1)$.
 
 However, we can replace (3) by (3$'$) plus an extra condition, which gives us a second version of the structure theorem:
 
\end{Rem}

 \begin{Thm}[structure theorem, second version]\label{secondstructure}
  Let $(G,\precsim)$ be a q.o.a.g.
  Then $\precsim$ is compatible with $+$ if and only if $G$ admits a subgroup $H$ satisfying the following properties:
 \begin{enumerate}[(1)]
     \item $H$ is an initial segment of $G$.
   \item $(H,\precsim)$ is an ordered abelian group.
   \item[(3$'$)] The q.o $\precsim$ induces a valuational q.o on $G/H$ via the formula:
   \[g+H\precsim h+H\Leftrightarrow g-h\in H\vee(g-h\notin H\wedge g\precsim h)\]
   \item[(4)] For any $g,h\in G$, $g\notin H$ and $g-h\in H$ implies $g\sim h$.
   \end{enumerate}

 \end{Thm}
 
 \begin{proof}
    If $\precsim$ is compatible, then (1)+(2)+(3$'$)+(4) holds with $H:=G^o$ ((3$'$) is Proposition \ref{quotientGo} and (4) is 
   Remark \ref{sameHclassareequivalent}(3)).
    Now assume that (1)+(2)+(3$'$)+(4) holds.   We just have to show that 
    (3) of Theorem \ref{structure} holds. By (3$'$), we know that the q.o
    $\precsim$ on $G/H$ is valuational, and we denote by $v:G/H\to\Gamma\cup\{\infty\}$ the corresponding valuation. 
    We lift $v$ to a valuation $w$ on $G$ as follows: add a point $\gamma_0$ to $\Gamma$ such that 
    $\Gamma<\gamma_0<\infty$. For any $g\in G$ define $w(g)$ as follows: 
    
    \[w(g)=\left\{\begin{array}{cc}
                 v(g+H) & \text{ if } g\notin H\\
                 \gamma_0 & \text{ if } 0\neq g\in H\\
                 \infty & \text{ if } g=0
                \end{array}\right.\]
    It is clear from its definition that 
    $w$ is a valuation (because $v$ is a valuation and $H$ is a subgroup of $G$). 
    Take $g,h\notin H$. If $g\precsim h$, then $g+H\precsim h+H$ hence 
    $v(g+H)\geq v(h+H)$ hence $w(g)\geq w(h)$. Conversely, if 
    $w(g)\geq w(h)$, then by definition of $w$ we must have 
    $h+H\precsim g+H$. If $h-g\notin H$ then it immediately follows from the definition of $\precsim$ on 
    $G/H$ that $h\precsim g$, and if $h-g\in H$ then it follows from (4) that $h\precsim g$. This shows that 
    (3) of Theorem \ref{structure} holds.
 \end{proof}

    We can reformulate Theorem \ref{secondstructure} into the language of exact sequences:  a compatible q.o.a.g is 
    an extension of a valued group by an ordered group:
 \begin{Thm}[Structure theorem, third version]\label{thirdstructure}
  Let $(G,\precsim)$ be a q.o.a.g. Then
  $\precsim$ is compatible with $+$ if and only if there exists an exact sequence
  $0\rightarrow G^o\overset{\iota}{\rightarrow} G\overset{\pi}{\rightarrow} F\rightarrow 0$
  such that there exists a group order $\leq_o$ on $G^o$ and a valuation $v$ on 
  $F$ such that for any $g,h\in G$, \newline  
      $g\precsim h\Leftrightarrow (g,h\in \iota(G^o)\wedge \iota^{-1}(g)\leq_o \iota^{-1}(h))\vee(h\notin \iota(G^o)\wedge v(\pi(g))\geq v(\pi(h)))$
 \end{Thm}

 \begin{proof}
  One direction is given by Proposition \ref{formulafortheqo}. For the other direction, assume that such an exact sequence as above 
  exists. We use Theorem \ref{secondstructure}. Clearly, setting $H:=\iota(G^o)$ then 
  (2) of Theorem \ref{secondstructure} is satisfied. If $g\precsim h$ we have by assumption 
  $(g,h\in \iota(G^o)\wedge \iota^{-1}(g)\leq_o \iota^{-1}(h))\vee(h\notin \iota(G^o)\wedge v(\pi(g))\geq v(\pi(h)))$, 
  so $h\in H$ implies $g\in H$, which proves (1). We now show that 
  $\precsim$ induces a q.o on $G/H\cong F$. If $g_1\precsim g_2$  and $g_1-g_2\notin H$, 
  then $g_2\notin H\wedge v(\pi(g_1))\geq v(\pi(g_2)))$, and since $H=\ker\pi$ we have for every $h_1,h_2\in H$, 
  $\pi(g_1+h_1)=\pi(g_1),\pi(g_2+h_2)=\pi(g_2)$ hence $g_2+h_2\notin H\wedge v(\pi(g_1+h_1))\geq v(\pi(g_2+h_2)))$
  so $g_1+h_1\precsim g_2+h_2$. This shows that the condition of Lemma \ref{quotient} is satisfied. It is then easy to see that 
  the q.o induced by $\precsim$ on $G/H\cong F$ is exactly the q.o corresponding to $v$. This shows (3$'$). Finally, if 
  $g\notin H$ and $g-h\in H$, then we have 
  $g,h\notin H\wedge v(\pi(g))=v(\pi(h))$ hence by assumption $g\sim h$, hence (4).
 \end{proof}

 \section{Products of compatible q.o's}\label{productsection}
 In the theory of ordered abelian groups, there is a natural notion of 
   product, namely the lexicographic product (see Section \ref{prelsection}).
   The goal of this section is to develop a similar notion for compatible q.o.a.g's.  
   We first introduce the notion of compatible Hahn product, and we then use this notion to prove a
   generalization of Hahn's embedding theorem for compatible q.o.a.g's. We then show that 
   the compatible product of an ordered group by a valued group preserves elementary equivalence.
 \subsection{The compatible Hahn product}
   
    In Section \ref{prelsection}, we recalled the definition of the lexicographic product of a family of 
    ordered groups. Unfortunately, we cannot generalize this definition to q.o groups by simply replacing 
    $\leq$ by a q.o. Indeed, with such a definition, the lexicographic product of a family of 
    valuational q.o's would not be a compatible q.o. This forces us to introduce a specific notion of product for
    valuational q.o's.   
  Given an ordered family $(B_{\gamma},\precsim_{\gamma})_{\gamma\in\Gamma}$ of q.o groups, let $G:=\Hahn{B}$ and let $v$ be the usual valuation of 
  $G$, i.e $v(g)=\min\supp(g)$.  
  We define 
    the \textbf{valuational Hahn product} of the family $(B_{\gamma},\precsim_{\gamma})_{\gamma\in\Gamma}$  as the quasi-ordered group $(G,\qoval)$,
    where 
    $\qoval$ is defined as follows: \newline    
    $g=(g_{\gamma})_{\gamma}\qoval h=(h_{\gamma})_{\gamma}\Leftrightarrow g_{\delta}\precsim_{\delta} h_{\delta}$, where 
    $\delta=\min(v(g),v(h))$. We then have the following:
 
 \begin{Prop}\label{valprod}
      Let $(B_{\gamma},\precsim_{\gamma})_{\gamma\in\Gamma}$ be an ordered family of valuational q.o's and let 
      $(G,\qoval)$ be the valuational product of the family $(B_{\gamma},\precsim_{\gamma})_{\gamma\in\Gamma}$. Then $\qoval$ is valuational.      
      \end{Prop}

\begin{proof}
     Let $g\in G$, $g\neq0$, and $\delta:=v(g)$. Since $\precsim_{\delta}$ is valuational, 
     we have $0\precnsim_{\delta}g_{\delta}$ and $g_{\delta}\sim_{\delta}-g_{\delta}$, which implies by definition of 
     $\precsim$ that $0\nqoval g$ and $g\eqval -g$. We just have to verify that $\qoval$ satisfies the ultrametric inequality. 
     Let $h\in G$ with $h\qoval g$, which implies  in particular   $v(g)\leq v(h)$ and $h_{\delta}\precsim_{\delta} g_{\delta}$ 
     where $\delta=v(g)$.
     Since $\precsim_{\delta}$ is 
     valuational, $h_{\delta}\precsim g_{\delta}$ implies  $h_{\delta}+g_{\delta}\precsim_{\delta}g_{\delta}$. 
     Moreover, $v(g+h)\geq\min(v(g),v(h))$, hence $\delta=\min(v(g),v(g+h))$. It follows from the definition of $\qoval$ that 
     $g+h\qoval g$.
\end{proof}

 This allows us to define a notion of  product for compatible q.o.a.g's. 
 Let $(B_{\gamma},\precsim_{\gamma})_{\gamma\in\Gamma}$ be an ordered family of compatible q.o.a.g's. Let $G:=\Hahn{B}$, let
 $(G^o,\leq_o)$ be the lexicographic product of the family $(B_{\gamma}^o,\precsim_{\gamma})_{\gamma\in\Gamma}$ (seen as ordered groups) and set 
 $G^v:=G\backslash G^o$. For each $\gamma\in\Gamma$, the q.o $\precsim_{\gamma}$ induces a valuational q.o on $B_{\gamma}/B_{\gamma}^o$ by 
 Proposition \ref{quotientGo}.
 Let $(F,\qoval)$ be the valuational Hahn product of the family 
 $(B_{\gamma}/B_{\gamma}^o,\precsim_{\gamma})_{\gamma\in\Gamma}$.  Note that $G/G^o$ is canonically isomorphic to 
 $F$ via the isomorphism $\psi: (g_{\gamma})_{\gamma}+G^o\mapsto (g_{\gamma}+B_{\gamma}^o)_{\gamma}$.
 We  define 
 the \textbf{compatible Hahn product} of the family $(B_{\gamma},\precsim_{\gamma})_{\gamma\in\Gamma}$ as the 
 compatible q.o.a.g $(G,\precsim)$, where $\precsim$ is defined by the following formula: 
     $g\precsim h\Leftrightarrow (g,h\in G^o\wedge g\leq_o h)\vee(h\in G^v\wedge \psi(g+G^o)\qoval \psi(h+G^o))$. 
     The fact that $\precsim$ is compatible follows directly from Theorem \ref{thirdstructure}. We denote the compatible 
     Hahn product of the family $(B_{\gamma},\precsim_{\gamma})_{\gamma\in\Gamma}$ by 
     $\text{H}_{\gamma\in\Gamma}(B_{\gamma},\precsim_{\gamma})$.
     
     One particular case of compatible Hahn product is the case of the product of an ordered group by a 
     valued group. If $(G,\leq)$ is an ordered group and $(F,\qoval)$ a group endowed with a valuational q.o, then 
     we denote their compatible Hahn product by 
     $(G,\leq)\cprod (F,\qoval)$. 
    This gives us a way of constructing compatible q.o.a.g's from ordered and valued groups. 
   In particular, we have the following: 
   
    \begin{Prop}\label{realizingorderedandvaluedpart}
    Let $(G,\leq)$ be an ordered abelian group and $(F,v)$ a valued group. Then there exists 
    a compatible q.o.a.g whose ordered part is $(G,\leq)$ and whose valued part is $(F,v)$.
   \end{Prop}
   
   \begin{proof}
    Just take  $(G,\leq)\cprod(F,\qoval)$, where $\qoval$ is the q.o corresponding 
    to $v$.    
   \end{proof}

    In view of Theorem \ref{secondstructure}, it is natural to ask whether every compatible 
    q.o.a.g 
    can be obtained as the product of an ordered group by a valued group. However,
    Example \ref{contreexemple}(b) shows that it is not the case: 
    $G^o=5\Z$ is not a direct factor of $G=\Z$, so $(G,\precsim)$ is not the compatible product of 
     $G^o$ by $G/G^o$.      
     Fortunately, we have the following:
    \begin{Prop}\label{Gascompatibleproduct}
   Let $(G,\precsim)$ be a compatible q.o.a.g. If $G^o$ is a direct summand of $G$ with complement $F$, then $(F,\precsim)$ is canonically isomorphic to the valued part of 
   $(G,\precsim)$ and we have $(G,\precsim)=(G^o,\precsim)\cprod(F,\precsim)$. In other words, $(G,\precsim)$ is the compatible Hahn product of its 
   ordered part by its valued part.
  \end{Prop}
    \begin{proof}
     It follows directly from Proposition \ref{formulafortheqo} and from the definition of the compatible Hahn product.
    \end{proof}

    Proposition \ref{realizingorderedandvaluedpart} shows that if $(G,\precsim)$ is a compatible q.o.a.g, then the valuation appearing in the 
    valued part of $G$ can a priori be any valuation, in particular it does not have to be a $\Z$-module valuation. However, 
    we mentioned in the introduction that it can be interesting to restrict our attention to 
    such valuations. Consider the following family of axioms indexed by $n\in\N$:
       $(VM_n)\quad\forall g, -g\precsim g\Rightarrow g\precsim ng$ (``VM'' stands for ``valued module''). 
    This family of axioms gives an axiomatization of the class of compatible q.o.a.g's whose valuation on its valuational part is 
    a $\Z$-module valuation. If $(G,\precsim)$ is such a compatible q.o.a.g, then $G^o$ is pure in $G$, from which we get the
    following result: 
    
    \begin{Prop}
     Let $(G,\precsim)$ be a compatible q.o.a.g satisfying the axiom $(VM_n)$ for every $n\in\N$.
     Assume that $G$ is divisible. Then $(G,\precsim)$ is the compatible Hahn product of its ordered part by its valued part.
    \end{Prop}

     \begin{proof}
      Because of $(VM_n)$, $G^o$ is pure in $G$. Since $G$ is divisible, $G^o$ is then a direct summand of $G$. 
      The result then follows from Proposition \ref{Gascompatibleproduct}.
     \end{proof}

 \subsection{Hahn's embedding theorem}\label{Hahntheoremsection}
  
  We now want to generalize Hahn's embedding theorem for ordered groups to quasi-ordered groups.
  This implies defining a notion of archimedeanity for q.o groups. To do this, we will associate 
  a valuational q.o $\archqo$ to each compatible q.o $\precsim$, which we will call 
  the archimedean q.o associated to $\precsim$.
  
  Let $(G,\precsim)$ be a compatible q.o.a.g.
  Consider the relation $\preqo$ defined as follows: we say that 
  $g\preqo h$ if and only if there is $n,m\in\Z\backslash\{0\}$ such that 
  $0\precsim ng\precsim mh$. We have the following:
  \begin{Lem}\label{preqotransonGo}
   The relation $\preqo$ defines a valuational q.o on $G^o$.
  \end{Lem}
  \begin{proof}
   $\preqo$ is clearly reflexive, let us show transitivity. Assume $f\preqo g\preqo h$. There are 
   $n,m,k,l$ with $0\precsim nf\precsim mg$ and $0\precsim kg\precsim lh$. It follows that 
   $0\precsim \vert n\vert\vert f\vert\precsim \vert m\vert \vert g\vert$ and   
   $0\precsim \vert k\vert\vert g\vert\precsim \vert l\vert \vert h\vert$.
   Since $\precsim$ is an order on $G^o$ these relations imply 
   $0\precsim \vert n\vert \vert f\vert\precsim \vert m\vert \vert g\vert\precsim \vert km\vert \vert g\vert\precsim \vert ml\vert\vert h\vert$, so 
   either $0\precsim nf\precsim mlh$ or $0\precsim nf\precsim -mlh$ holds,
   hence $f\preqo h$. This proves that $\preqo$ is a q.o on $G^o$. Now let us prove that 
   $\preqo$ is valuational. If $g\preqo 0$ then $0\precsim ng\precsim 0$ holds for some $n\neq0$ which implies $ng=0$ by $(Q_1)$, and since 
   $G^o$ is an ordered abelian group it is torsion-free, hence $g=0$. Thus, we have 
   $0\precnsim g$ for every $g\in G^o$ with $g\neq0$. Clearly, $g\preqo -g\preqo g$ holds for every $g$.
   Now assume $g\preqo h$ and take $n,m$ with 
   $0\precsim ng\precsim mh$. We have 
   $0\precsim \vert n\vert\vert g\vert\precsim\vert m\vert\vert h\vert$, which by compatibility implies 
   $\vert n\vert\vert g+h\vert\precsim\vert n+m\vert\vert h\vert$ hence
   $g+h\preqo h$. This proves that the ultrametric inequality holds. 
  \end{proof}

  However, the relation $\preqo$ is not transitive in general. Indeed, consider the following example: 
  set $G:=\Z^2$ and let $v: G\to\{1,2,3,\infty\}$ be the valuation defined as follows: 
  
  \[v(n,m)=\begin{cases}
            1 \text{ if }  p\nmid m\\
            2 \text{ if } 0\neq n\wedge p\mid m\\
            3 \text{ if } n=0\wedge p\mid m\neq0\\
            \infty \text{ if } n=m=0
           \end{cases}\]

           Now let $\precsim$ be the q.o induced by $v$.
           Let $f:=(0,p)$, $g:=(1,p)$ and $h:=(0,1)$. We have 
           $p.h\sim f$, hence $h\preqo f$. Moreover, $g\precsim h$, so $g\preqo h$. However, 
           for every $n,m\in\Z\backslash\{0\}$ we have $mf\sim f\precnsim g\sim ng$, so 
           $g\preqo f$ does not hold. To make $\preqo$ transitive, we define the relation 
           $\archqo$ as follows: we say that $g\archqo h$ if there exists 
           $r\in\N$ and $x_1,\dots,x_r\in G$ such that 
           $g\preqo x_1\preqo x_2\preqo\dots\preqo x_r\preqo h$. Note that $\archqo$ is the same as 
           $\preqo$ for ordered groups. In order to prove that $\archqo$ is a valuational q.o, we need the following 
           lemma:
           
           \begin{Lem}\label{coarseningofvalisval}
            Let $\precsim$ be a compatible q.o, $g\in G\backslash\{0\}$ v-type and let 
            $\etoile$ be a coarsening of $\precsim$. Then for any $h\in G$, 
            $h\etoile g$ implies $g+h\etoile g$. In particular, if $\precsim$ is valuational then 
            $\etoile$ is also valuational.
           \end{Lem}

           \begin{proof}
            Assume $h\etoile g$ and $g\netoile g+h$. Since $\etoile$ is a coarsening of 
            $\precsim$, this implies $g\precnsim g+h$. Since $g$ is v-type, $g+h$ must also be v-type by Proposition 
            \ref{segin}, and it then follows from Proposition \ref{ultrametric} that 
            $g+h\sim h$. Since $\etoile$ is a coarsening of $\precsim$, this implies 
             $h\sim^{\ast} g+h$. We thus have 
            $h\etoile g\netoile g+h\sim^{\ast} h$, which is a contradiction.
           \end{proof}

           \begin{Prop}\label{archvaliscoarening}
            Let $(G,\precsim)$ be a torsion-free compatible q.o.a.g. 
            The relation $\archqo$ is a q.o on $G$. Moreover, it is the finest $\Z$-module-valuational coarsening of 
            $\precsim$.
           \end{Prop}

           \begin{proof}
            The fact that $\archqo$ is transitive and is a coarsening of $\precsim$ is clear from its definition. 
            It is also clear that $g\archeq ng$ for all $g\in G$ and $n\in\Z\backslash\{0\}$. Now let us show that 
            $\archqo$ is valuational. Note that for any $g\in G$, 
            $g\preqo 0$ implies $g=0$: indeed, if $g\preqo 0$ then there exists 
            $n\neq0$ with $0\precsim ng\precsim 0$, which by $(Q_1)$ implies $ng=0$ and since $G$ is torsion-free it follows that 
            $g=0$. By definition of $\archqo$, it then follows that $g\archqo 0$ implies $g=0$, so we have 
            $0\narchqo g$ whenever $g\neq0$.
                        Now let us show that $\archqo$ satisfies the ultrametric inequality. Let $g,h\in G$ with 
            $g\archqo h$. If $h\in G^v$, it follows from Lemma \ref{coarseningofvalisval} that 
            $g+h\archqo h$, so assume $h\in G^o$. If $g\in G^o$, it follows from 
            Lemma \ref{preqotransonGo} that $g+h\archqo h$. Assume then that $g\in G^v$. By Proposition 
            \ref{segin} we then have $h\precnsim g$, hence by Proposition \ref{ultrametric} $g+h\sim g$, and since 
            $\archqo$ is a coarsening of $\precsim$ this implies $g+h\archeq g$, hence 
            $g+h\archqo h$. This proves that $\archqo$ is valuational. 
            
            Now let $\etoile$ be another coarsening of $\precsim$ such that $\etoile$ is 
            $\Z$-module valuational. We show that $\etoile$ is a coarsening of $\archqo$. 
            Let $g,h\in G$ with $g\archqo h$. We first assume that
            $g\preqo h$. There is $n,m$ with $0\precsim ng\precsim mh$. Since $\etoile$ is a coarsening of $\precsim$, 
            this implies $ng\etoile mh$. Since $\etoile$ is $\Z$-module-valuational, we have 
            $ng\sim^{\ast} g$ and $mh\sim^{\ast} h$, hence $g\etoile h$. Now for the general case, 
            we know that there are $x_1,\dots,x_r$ with $g\preqo x_1\preqo\dots\preqo x_r\preqo h$. By what we just proved 
            this implies $g\etoile x_1\etoile\dots\etoile x_r\etoile h$, which by transitivity of $\etoile$ implies 
            $g\etoile h$.This shows that $\etoile$ is a coarsening of $\archqo$.
           \end{proof}

  The valuation $\archval$ corresponding to $\archqo$ is called \textbf{the archimedean valuation associated to} $\precsim$.
  $(G,\precsim)$ is an \textbf{archimedean compatible q.o.a.g} if $\archval$ is the trivial valuation on $G$.
  If $(\Gamma,(B_{\gamma})_{\gamma\in\Gamma})$ is the skeleton of $(G,\archval)$, note that each $G_{\gamma}$ is $\precsim$-convex 
  (this follows from the fact that $\archqo$ is a coarsening of $\precsim$). By proposition 
  \ref{convexsubgroupProp}, it follows that $\precsim$ naturally induces a compatible q.o $\precsim_{\gamma}$ on $B_{\gamma}$; 
  note that $(B_{\gamma},\precsim_{\gamma})$ is archimedean. 
  In order to state our Hahn's embedding theorem for q.o groups, we first need to 
  strengthen the embedding theorem for valued $\Z$-modules given in Section \ref{prelsection}:
  
  \begin{Thm}\label{improvedHahnembeddingforvaluedgroups}
   Let $(G,v)$ be a group endowed with a $\Z$-module valuation, $(\Gamma,(B_{\gamma})_{\gamma\in\Gamma})$ the skeleton 
   of $(G,v)$, $H:=\Hahn{B}$ and $w$ the usual valuation on $H$, i.e 
   $w(h)=\min\supp(h)$. There exists a group embedding 
   $\phi:G\to\Hahn{B}$ such that the following holds:
   \begin{enumerate}
    \item For any $g\in G$, $w(\phi(g))=v(g)$.
    \item For any $\gamma\in\Gamma$ and $g\in G$ with $v(g)=\gamma$, the coefficient of $\phi(g)$ at $\gamma$ is
    $g+G_{\gamma}$.
   \end{enumerate}

  \end{Thm}
  \begin{proof}
   Assume the theorem has been proved for divisible groups and let $\hat{G}$ be the divisible hull of $G$. It is easy to see that there is a unique way of 
   extending $v$ to a $\Z$-module valuation on $\hat{G}$ and that $(\hat{G},v)$ has the same skeleton as 
   $(G,v)$. Now if $\iota$ is the natural embedding from $G$ to $\hat{G}$ and $\phi:\hat{G}\to H$ is as in the 
   theorem, then $\phi\circ\iota:G\to H$ also satisfies the conditions of the theorem. Thus, we can assume that $G$ is divisible. 
   We know from \cite{SKuhlmann} that there exists 
   a group embedding $\psi: G\to H$ and an isomorphism of ordered set $\lambda: \Gamma\to\Gamma$ such that 
   $w(\psi(g))=\lambda(v(g))$ for every $g\in G$. Now consider 
   $\chi: H\to H, (h_{\gamma})_{\gamma\in\Gamma}\mapsto (h_{\lambda(\gamma)})_{\gamma\in\Gamma}$. One can easily check 
   that $\chi$ is a group isomorphism and that 
   $\chi\circ\psi:G\to H$ satisfies condition (1) of the theorem. Therefore, we can assume that $\psi$ satisfies (1). 
   For every $\gamma\in\Gamma$, consider now 
   $\epsilon_{\gamma}:B_{\gamma}\to B_{\gamma}, g+G_{\gamma}\mapsto \phi(g)_{\gamma}$, where $\phi(g)_{\gamma}$ denotes the coefficient of 
   $\phi(g)$ at $\gamma$. For each $\gamma$, $\epsilon_{\gamma}$ is well-defined: indeed, 
   if $g,h\in G^{\gamma}$ are such that $v(g-h)>\gamma$, then since $\phi$ satisfies condition (1) we have 
   $w(\phi(g)-\phi(h))=v(g-h)>\gamma$ so $\phi(g)_{\gamma}=\phi(h)_{\gamma}$.
   One easily sees that 
   $\epsilon_{\gamma}$ is a group isomorphism. Define 
   $\zeta:H\to H,(h_{\gamma})_{\gamma\in\Gamma}\mapsto (\epsilon_{\gamma}^{-1}(h_{\gamma}))_{\gamma\in\Gamma}$. 
   $\zeta$ is a group isomorphism, and it is easy to see that 
   $\zeta\circ\psi:G\to H$ satisfies all the conditions of the theorem.
  \end{proof}

We can now state 
  a Hahn's embedding theorem for compatible q.o.a.g's:
  
  \begin{Thm}\label{Hahnforqogroups}
   Let $(G,\precsim)$ be a torsion-free compatible q.o.a.g. and let $\archval$ be the archimedean valuation associated to 
   $\precsim$. Let $(\Gamma,(B_{\gamma})_{\gamma\in\Gamma})$ be the skeleton of $(G,\archval)$ and let 
   $\precsim_{\gamma}$ be the q.o induced by $\precsim$ on $B_{\gamma}$. Then there exists an embedding of quasi-ordered groups 
   from $(G,\precsim)$ into the compatible Hahn product $\Hahn{(B_{\gamma},\precsim_{\gamma})}$.
  \end{Thm}

  \begin{proof}
   Let $(H,\etoile)$ denote the compatible Hahn product of the family 
   $(B_{\gamma},\precsim_{\gamma})_{\gamma\in\Gamma}$ and let $v$ denote the usual valuation 
   on $H$ (i.e $v(h)=\min\supp(h)$). We denote by $(F,\qoval)$ the 
    valuational product of the family $(B_{\gamma}/B^o_{\gamma})_{\gamma\in\Gamma}$.    
    We take a group embedding $\phi:G\to H$ as given by Theorem \ref{improvedHahnembeddingforvaluedgroups}.
    For $g\in G$ we denote by $g_{\gamma}$ the coefficient of $\phi(g)$ at $\gamma$. 
    We need the following claims:
   
   \begin{Claim}\label{claimone}
    For any $\gamma\in\Gamma$ and $g\in G$ with $\archval(g)=\gamma$, $g$ is v-type if and only if 
    $g+G_{\gamma}$ is v-type.
   \end{Claim}
\begin{proof}
 It follows from Remark \ref{remarkonquotientqo}(4).
\end{proof}
   
   \begin{Claim}\label{Claimtwo}
    If $h\in G$ is o-type and $\delta=\archval(h)$, then for every 
    $\gamma>\delta$, $B_{\gamma}^o=B_{\gamma}$.    
   \end{Claim}

   \begin{proof}
      Let $g\in G$ with $\archval(g)=\gamma$.  Since $\archval(g)>\archval(h)$ and since $\archqo$ is a coarsening of 
      $\precsim$, we must have $g\precsim h$. By Proposition \ref{segin}, it follows that $g$ is o-type.
      By Claim \ref{claimone}, $g+G_{\gamma}$ is then o-type. This shows that every element of $B_{\gamma}$ 
      is o-type.
   \end{proof}

  \begin{Claim}\label{claimthree}
   For any $g\in G$ and $\delta:=\archval(g)$, we have 
   $\min\supp((g_{\gamma}+B_{\gamma}^o)_{\gamma\in\Gamma})\geq\delta$. If moreover $g\in G^v$, then we have 
   equality.
  \end{Claim}
  \begin{proof}
   By condition (1) of Theorem \ref{improvedHahnembeddingforvaluedgroups}, we have 
   $\delta=v(\phi(g))$. It then follows from the definition of 
   $v$ and $\delta$ that $g_{\epsilon}=0$ for every $\epsilon<\delta$, hence 
   $g_{\epsilon}+B_{\epsilon}^o=0$ for every $\epsilon<\delta$ hence 
   $\min\supp((g_{\gamma}+B_{\gamma}^o)_{\gamma\in\Gamma})\geq\delta$. Now if $g\in G^v$, then 
   by Claim \ref{claimone} $g+G_{\delta}$ is v-type, and by condition (2) of 
   Theorem \ref{improvedHahnembeddingforvaluedgroups} we have 
   $g_{\delta}=g+G_{\delta}$, hence $g_{\delta}\notin B_{\delta}^o$ hence 
   $\delta\in\supp((g_{\gamma}+B_{\gamma}^o)_{\gamma\in\Gamma})$, which proves 
   $\delta=\min\supp((g_{\gamma}+B_{\gamma}^o)_{\gamma\in\Gamma})$.   
  \end{proof}

    Now let us show the theorem.  
   Take $g,h\in G$ and set $\alpha:=\archval(g)=v(\phi(g))$ and
   $\beta:=\archval(h)=v(\phi(h))$. We want to show that 
   $g\precsim h\Leftrightarrow\phi(g)\etoile\phi(h)$. Without loss of generality, we assume $g\neq h$.
   We first assume that 
    $g\precsim h$ and we show that $\phi(g)\etoile\phi(h)$.
    Note that 
   $g\precsim h$ implies $\alpha\geq\beta$ by Proposition \ref{archvaliscoarening}. 
   We use the formula of Proposition \ref{formulafortheqo}. We first 
   consider the case $h\in G^o$, which implies $g\in G^o$ by Proposition \ref{segin}. 
   Since $g,h$ are o-type, it follows from Claim \ref{claimone} and from condition (2) of Theorem 
   \ref{improvedHahnembeddingforvaluedgroups} that $g_{\alpha}$ and $h_{\beta}$ are o-type. Moreover, it follows from 
   Claim \ref{Claimtwo} that 
    $B_{\epsilon}^o=B_{\epsilon}$ for every $\epsilon>\beta$. It follows that $g_{\epsilon},h_{\epsilon}\in B_{\epsilon}^o$ for every 
   $\epsilon\geq\beta$, and since $\archval(g),\archval(h)\geq\beta$  this implies that 
   $\phi(g)$ and  $\phi(h)$ both lie in $H^o=\Hahn{B^o}$. Now set $\epsilon:=\archval(g-h)=v(\phi(g)-\phi(h))$.
   Since $g\precsim h\in G^o$, we have 
   $0\precsim h-g$, hence 
   $0\precsim_{\epsilon}(h-g)+G_{\epsilon}$, and 
   $(h-g)+G_{\epsilon}=(h-g)_{\epsilon}$ by condition (2) of \ref{improvedHahnembeddingforvaluedgroups}.
   We thus have 
   $0\precsim h_{\epsilon}-g_{\epsilon}$, and since 
   $g_{\epsilon},h_{\epsilon}\in B_{\epsilon}^o$ this implies
   $g_{\epsilon}\precsim_{\epsilon}h_{\epsilon}$.
    By definition of $\etoile$ on $H^o$, this implies 
   $g\etoile h$.   
   Now consider the case where $h\in G^v\wedge g+G^o\precsim h+G^o$. It follows from Claim \ref{claimone} that
   $h_{\beta}$ is v-type, hence $\phi(h)\notin H^o$. Now note that 
   $g_{\beta}\precsim h_{\beta}$: indeed, if $\alpha>\beta$ then 
   $g_{\beta}=0$, and since $h_{\beta}$ is v-type we have $0\precsim_{\beta}h_{\beta}$; if 
   $\beta=\alpha$, we have $g_{\beta}=g+G_{\beta}$ and $h_{\beta}=h+G_{\beta}$, and $g\precsim h$ implies  
   $g+G_{\beta}\precsim_{\beta}h+G_{\beta}$ by Remark \ref{remarkonquotientqo}(4), hence
   $g_{\beta}\precsim h_{\beta}$. By Remark \ref{remarkonquotientqo}(4), 
   $g_{\beta}\precsim h_{\beta}$ implies $g_{\beta}+B_{\beta}^o\precsim h_{\beta}+B_{\beta}^o$. 
   Moreover, Claim \ref{claimthree} implies that 
    $\beta=\min(\min\supp((h_{\gamma}+B_{\gamma}^o)_{\gamma\in\Gamma}),\min\supp((g_{\gamma}+B_{\gamma}^o)_{\gamma\in\Gamma}))$.
   It then follows from the definition 
   of $\qoval$ that $((g_{\gamma}+B_{\gamma}^o)_{\gamma\in\Gamma})\qoval((h_{\gamma}+B_{\gamma}^o)_{\gamma\in\Gamma})$, which by definition of 
   $\etoile$ implies $g\etoile h$.

   This shows $g\precsim h\Rightarrow\phi(g)\precsim\phi(h)$, let us show the converse. 
   Assume that $\phi(g)\etoile\phi(h)$ holds and let us show that $g\precsim h$.
   Note that if $\beta<\alpha$, then 
   since $\archval$ is a coarsening of $\precsim$ we have 
   $g\precnsim h$, so we can assume $\beta\geq\alpha$.
   We first consider the case
   $\phi(h)\in H^o$, which implies $\phi(g)\in H^o$ by Proposition \ref{segin}. By definition of 
   $\etoile$ on $H^o$, we have $g_{\gamma}\precnsim_{\gamma}h_{\gamma}$ for $\gamma=v(\phi(g)-\phi(h))$, and since 
   $h_{\gamma},g_{\gamma}\in B_{\gamma}^o$ this implies $0\precnsim_{\gamma}(h-g)_{\gamma}$. Since 
   $\gamma=v(\phi(g)-\phi(h))$, we have 
   $(h-g)_{\gamma}=(h-g)+G_{\gamma}$ by condition (2) of Theorem \ref{improvedHahnembeddingforvaluedgroups}, hence 
   $0\precnsim_{\gamma}(h-g)_{\gamma}$,
   which by Remark 
   \ref{remarkonquotientqo}(4) implies $0\precnsim h-g$. By Claim \ref{claimone}, 
   $h-g$ is o-type, so the previous inequality implies $g\precsim h$.
   Now assume 
   that $\phi(h)\in H^v$.  By definition of $\etoile$, this implies 
   $(g_{\gamma}+B_{\gamma}^o)_{\gamma\in\Gamma}\qoval(h_{\gamma}+B_{\gamma}^o)_{\gamma\in\Gamma}$.
   By definition of $\qoval$, this implies 
   $\min(\supp((h_{\gamma}+B_{\gamma}^o)_{\gamma\in\Gamma}))\leq\min(\supp((g_{\gamma}+B_{\gamma}^o)_{\gamma\in\Gamma}))$, 
   which by Claim \ref{claimthree} implies 
   $\beta\leq \alpha$, hence 
   $\beta=\alpha=\min(\min(\supp((h_{\gamma}+B_{\gamma}^o)_{\gamma\in\Gamma})),\min(\supp((g_{\gamma}+B_{\gamma}^o)_{\gamma\in\Gamma}))))$.
    The definition of 
   $\qoval$ then also implies  $g_{\beta}+B_{\beta}^o\precsim_{\beta}h_{\beta}+B_{\beta}^o$. 
   Since $h_{\beta}\notin B_{\beta}^o$, this inequality 
   implies 
   $g_{\beta}\precsim h_{\beta}$ by Remark \ref{remarkonquotientqo}(4).
   Since $\beta=\archval(g)=\archval(h)$, condition (2) of Theorem 
   \ref{improvedHahnembeddingforvaluedgroups} implies 
   $g_{\beta}=g+G_{\beta}$ and $h_{\beta}=h+G_{\beta}$, hence 
   $g+G_{\beta}\precsim_{\delta}h+G_{\beta}$, which by Remark 
   \ref{remarkonquotientqo}(4) implies $g\precsim h$. This finishes the proof.
   
  \end{proof}

  \subsection{Elementary equivalence and products}\label{Fefermanvaughtsection}
  
    In view of Theorem \ref{secondstructure}, it is natural to ask whether elementary equivalence of two compatible q.o.a.g's is equivalent to
    the elementary equivalence of their respective ordered parts and the elementary equivalence of their valued parts. This is the object of this 
    subsection.     We first show that one implication is always true : if two compatible q.o.a.g's are elementarily equivalent, then 
   so are their ordered parts and so are their valued parts (Proposition \ref{elementaryequivalenceofparts}). 
    The converse fails in general (see example \ref{notproductexample}), 
    but we show then that is holds for groups which are obtained as the compatible product of their ordered parts
    by their valued parts. In other words, we show that the compatible Hahn product of ordered groups by valued groups preserves 
    elementary equivalence (Theorem \ref{Fefermanvaughttheorem}).

    We let $\langue$ denote the language of quasi-ordered groups: $\langue=\{0,+,-,\precsim\}$ ($-$ is interpreted as a unary relation). 
    Note that the atomic formulas are all formulas of the 
form $P(\bar{x})\precsim Q(\bar{x})$ or $P(\bar{x})=0$, where 
$P(\bar{x}),Q(\bar{x})$ are expressions of the form $\sum_{i=1}^kn_ix_i$ with $n_1,\dots,n_k\in\Z$.
Note also that 
for any compatible q.o.a.g $G$, $G^o$ is definable in $G$ by the formula $x=0\vee -x\nsim x$, which we will thus abbreviate as the 
formula $x\in G^o$. Finally, note that for any term $P(\bar{x})$ and for
every tuples $\bar{g},\bar{h}\subseteq G$ we have $P(\bar{g}+\bar{h})=P(\bar{g})+P(\bar{h})$ and 
$P(\bar{g}+G^o)=P(\bar{g})+G^o$.

  \begin{Lem}\label{reversefefermanvaughtlemma}
         
         Let $\phi(\bar{x})$ be a formula of $\langue$. Then there exists two formulas $\phi^o(\bar{x}),\phi^v(\bar{x})$
         in $\langue$, each of the same arity as $\phi$, such that, if
         $(G,\precsim)$ is a compatible q.o.a.g with ordered part $(G^o,\leq)$ and valued part $(H,\precsim)$, we have:         
         \begin{enumerate}[(i)]
          \item For any $\bar{g}_o\subseteq G^o$, $G^o\vDash\phi(\bar{g}_o)$ if and only if $G\vDash\phi^o(\bar{g}_o)$
          \item For any $\bar{g}_v\subseteq H$,   $H\vDash\phi(\bar{g}_v)$ if and only if for all $\bar{g}\subseteq G$,
           $\bar{g}+G^o=\bar{g}_v\Rightarrow G\vDash \phi^v(\bar{g})$ if and only if there exists $\bar{g}\subseteq G$ with 
           $\bar{g}+G^o=\bar{g}_v$ and
          $G\vDash \phi^v(\bar{g})$.
         \end{enumerate}

        \end{Lem}

        \begin{proof}
         For (i): write $\phi$ in prenex form: 
         $\phi(\bar{x})\equiv Q_1y_1\dots Q_ny_n\psi(\bar{y},\bar{x})$, where each $Q_i$ is a quantifier and 
         $\psi$ is quantifier-free.
         . Since $G^o$ is definable in $G$ 
         (by the formula $x=0\vee -x\nsim x$), we can define the formula 
         $\phi^o(\bar{x})\equiv Q_1 y_1\in G^o\dots Q_ny_n\in G^o\psi(\bar{y},\bar{x})$, and it is then easy to see that 
         $\phi^o$ has the desired property.
        
         For (ii): We proceed by induction on $\phi$. Assume first that $\phi$ is atomic. If 
         $\phi$ has the form $P(\bar{x})=0$, define $\phi^v(\bar{x})\equiv P(\bar{x})\in G^o$. Now
         assume that          
         $\phi(\bar{x})\equiv P(\bar{x})\precsim Q(\bar{x})$ and 
         define $\phi^v(\bar{x})$ as $(P(\bar{x})\in G^o\wedge Q(\bar{x})\in G^o)\vee(Q(\bar{x})\notin G^o\wedge \phi(\bar{x}))$.
         Assume that $H\vDash\phi(\bar{g}_v)$ and take $\bar{g}\subseteq G$ with 
         $\bar{g}+G^o=\bar{g}_v$. We have  $H\vDash P(\bar{g}_v)\precsim Q(\bar{g}_v)$. If 
         $Q(\bar{g}_v)=0$, then since $\precsim$ is valuational on $H$ we must have 
         $P(\bar{g}_v)=0$, hence 
         $G\vDash P(\bar{g})\in G^o\wedge Q(\bar{g})\in G^o$. If 
         $Q(\bar{g}_v)\neq0$ Remark \ref{remarkonquotientqo}(4) implies that 
         $P(\bar{g})\precsim Q(\bar{g})$. This shows that 
         $G\vDash \phi^v(\bar{g})$. Conversely, assume that there exists a 
         $\bar{g}\subseteq G$ such that $\bar{g}+G^o=\bar{g}_v$ and $G\vDash \phi^v(\bar{g})$. 
	 If $G\vDash(P(\bar{g})\in G^o\wedge Q(\bar{g})\in G^o)$ then 
         $P(\bar{g}_v)=Q(\bar{g}_v)=0$, so in particular 
         $H\vDash P(\bar{g}_v)\precsim Q(\bar{g}_v)$. If $G\vDash Q(\bar{g})\notin G^o\wedge \phi(\bar{g})$, then 
         Remark \ref{remarkonquotientqo}(4) implies that 
         $H\vDash P(\bar{g}_v)\precsim Q(\bar{g}_v)$. This shows that $H\vDash\phi(\bar{g}_v)$ and concludes the case where $\phi$ 
         is atomic.  Assume now that       
         $\phi\equiv\neg\psi$ and set $\psi^v:\equiv\neg\phi^v$. If $H\vDash\phi(\bar{g}_v)$, then 
         $H\nvDash\psi(\bar{g}_v)$, so by induction hypothesis we have 
         $G\nvDash\psi^v(\bar{g})$ for all  $\bar{g}\subseteq G$ with $\bar{g}+G^o=\bar{g}_v$, hence 
         $G\vDash\phi^v(\bar{g})$. Conversely, if there is $\bar{g}\subseteq G$ with $\bar{g}+G^o=\bar{g}_v$ and 
         $G\vDash\phi^v(\bar{g})$, then $G\nvDash\psi^v(\bar{g})$ which by induction hypothesis means 
         $H\nvDash\psi(\bar{g}_v)$ hence $H\vDash\phi(\bar{g}_v)$.         
         If $\phi\equiv\phi_1\wedge\phi_2$, one can easily show that 
          $\phi^v:\equiv\phi_1^v\wedge\phi_2^v$ satisfies the desired property and if $\phi\equiv\exists y\psi(y,\bar{x})$, 
         it is also easy to see that  $\phi^v\equiv \exists y\psi^v(y,\bar{x})$ is suitable.
        \end{proof}

        \begin{Prop}\label{elementaryequivalenceofparts}
         Let $(G_1,\precsim_1)$ and $(G_2,\precsim_2)$ be two compatible q.o.a.g's such that 
         $G_1\equiv G_2$. Then the ordered part of $G_1$ is elementarily equivalent to the ordered part of $G_2$ and the valued 
         part of $G_1$ is elementarily equivalent to the valued part of $G_2$.        
        \end{Prop}

        \begin{proof}
           Assume that 
         $G_1\equiv G_2$ holds and let $\phi$ be a sentence of $\langue$. Take $\phi^o,\phi^v$ as in Lemma \ref{reversefefermanvaughtlemma}. 
         If $G_1^o\vDash\phi$, then $G_1\vDash\phi^o$, hence by assumption $G_2\vDash\phi^o$, 
         hence by choice of $\phi^o$: $G_2^o\vDash\phi$. We could show similarly that 
         $H_1\vDash \phi$ implies $H_2\vDash\phi$, hence $G_1^o\equiv G_2^o$ and 
         $H_1\equiv H_2$.
        \end{proof}
        
     The next example shows that the converse of Proposition \ref{elementaryequivalenceofparts} is false in general:

        \begin{Ex}\label{notproductexample}
        Take $G_1:=\Z$ with ordered part $G_1^o:=5\Z$ with the usual order and valued part $H_1:=\Z/5\Z$ equipped with the trivial valuation. 
         Now take $G_2:=G_1^o\cprod H_1$. 
         Since $G_2$ has torsion and $G_1$ does not, it is clear that $G_1$ and $G_2$ are not elementarily equivalent.
        \end{Ex}
        
  However, the next Lemma shows that the converse of Proposition \ref{elementaryequivalenceofparts} is true if we restrict ourselves to 
  compatible q.o.a.g's which are obtained as the product of their ordered part by their valued part (which is not the case of $G_1$ in example 
  \ref{notproductexample}):
  
   \begin{Lem}\label{fefermanvaughtlemma}
       Let $\phi(\bar{x})$ be a formula of
       $\langue$. Then there is $n\in\N$ such that there are $2n$ formulas \newline
       $\phi_1^o(\bar{x}),\dots,\phi_n^o(\bar{x}),\phi_1^v(\bar{x}),\dots,\phi_n^v(\bar{x})$, each having the same arity as $\phi$, such that the following holds: 
       
       For any ordered abelian group $G^o$ and any valuationally quasi-ordered group $H$, 
     for any $\bar{g}=\bar{g}_o+\bar{g}_v$ in $G:=G^o\cprod H$, 
        we have: $G\models\phi(\bar{g})$ if and only if there exists $i\in\{1,\dots,n\}$ such that 
        $G^o\models\phi_i^o(\bar{g}_o)$ and $H\models\phi_i^v(\bar{g}_v)$.       
      \end{Lem}

        \begin{proof}
         We identify $H$ with $G/G^o$. Note that the q.o induced by $\precsim$ on $G/G^o$ coincides with the q.o of $H$.
         We proceed by induction on $\phi$. We first assume that $\phi$ is an atomic formula. 
         If $\phi$ is of the form 
         $P(\bar{x})=0$, set $n=1$ and $\phi_1^o\equiv\phi_1^v\equiv\phi$. Assume that $\phi$ is of the form $P(\bar{x})\precsim Q(\bar{x})$. Set $n=2$ and 
         define $\phi_1^o(\bar{x}):\equiv (\bar{x}=\bar{x})$, $\phi_1^v(\bar{x}):\equiv (Q(\bar{x})\neq0\wedge\phi(\bar{x}))$, 
         $\phi_2^o:\equiv\phi$ and $\phi_2^v(\bar{x}):\equiv (Q(\bar{x})=P(\bar{x})=0)$. We must check that these formulas satisfy the desired condition. 
         Note that for any  $G^o,H,\bar{g}$ as above, we have 
         $P(\bar{g})=P(\bar{g}_o)+P(\bar{g}_v)$ with $P(\bar{g}_o)\in G^o$ and $P(\bar{g}_v)\in H$, and in  particular we have 
         $P(\bar{g})+G^o=P(\bar{g}_v)$ and
         $P(\bar{g})\in G^o$ if and only if $P(\bar{g}_v)=0$. With this remark in mind, it follows directly from 
         Proposition \ref{formulafortheqo} that the formulas 
         $\phi_1^o,\phi_2^o,\phi_1^v,\phi_2^v$ satisfy the condition we want.
         This settles the case where $\phi$ is 
         atomic. 
              If 
         $\phi\equiv\psi\vee\chi$, and if $\psi_1^o,\dots,\psi_k^o,\psi_1^v,\dots,\psi_k^v,\chi_1^o,\dots,\chi_l^o,\chi_1^v,\dots,\chi_l^v$ are 
         the desired formulas for $\psi$ and $\chi$, we simply set 
         $n:=k+l$, $\phi_i^o:\equiv\psi_i^o,\phi_i^v:\equiv\psi_i^v$ for $1\leq i\leq k$ and 
         $\phi_i^o:\equiv\chi_i^o,\phi_i^v:\equiv\chi_i^v$ for $k<i\leq n$. 
         Now assume that $\phi\equiv\exists y\psi(y,\bar{x})$ and let $\psi_1^o,\dots,\psi_k^o,\psi_1^v,\dots,\psi_k^v$ be the desired formulas for 
         $\psi$. Define $n:=k$,     
         $\phi_i^o:\equiv\exists y\psi_i^o(y,\bar{x})$ and $\phi_i^v:\equiv\exists y\psi_i^v(y,\bar{x})$ for every $i\in\{1,\dots,n\}$.
         If $G\vDash\phi(\bar{g})$, then there is $h\in G$ with $G\vDash \psi(h,\bar{g})$, which implies by induction hypothesis that 
         there is $i$ with $G^o\vDash\psi_i^o(h_o,\bar{g}_o)$ and $H\vDash\psi_i^v(h_v,\bar{g}_v)$, hence 
         $G^o\vDash\phi_i^o(\bar{g}_o)$ and $H\vDash\phi_i^v(\bar{g}_v)$. Conversely, if we assume that 
         $G^o\vDash\phi_i^o(\bar{g}_o)$ and $H\vDash\phi_i^v(\bar{g}_v)$, then there is some 
         $h_o\in G^o$ and $h_v\in H$ with $G^o\vDash\psi_i^o(h_o,\bar{g}_o)$ and $H\vDash\psi_i^v(h_v,\bar{g}_v)$, and by induction hypothesis we then 
         have $G\vDash\psi(h_o+h_v,\bar{g})$ hence $G\vDash\phi(\bar{g})$. This shows that the formulas 
         $\phi_1^o,\dots,\phi_n^o,\phi_1^v,\dots,\phi_n^v$ have the desired property.

         Now we just have to consider the case 
         $\phi\equiv\neg\psi$. Let $\psi_1^o,\dots,\psi_k^o,\psi_1^v,\dots,\psi_k^v$ be given.
         Let $P:=\mathcal{P}(\{1,\dots,k\})$ denote the power set of $\{1,\dots,k\}$. 
         For any $I\in P$, we define $\phi_I^o,\phi_I^v$ as follows: 
         $\phi_I^o\equiv\bigwedge_{i\in I}\neg\psi_i^o$ and $\phi_I^v\equiv\bigwedge_{i\notin I}\neg\psi_i^v$. Now let us check that 
         the formulas $(\phi_I^o)_{I\in P}$ and $(\phi_I^v)_{I\in P}$ satisfy the desired property. 
         Assume that $G\vDash\phi(\bar{g})$, so $G\nvDash\psi(\bar{g})$. By induction hypothesis, this means that for all 
         $i\in\{1,\dots,k\}$, either $G^o\nvDash\psi_i^o(\bar{g}_o)$ or $H\nvDash\psi_i^v(\bar{g}_v)$. Choose $I\in P$ as the 
         set of all $i$ with $G^o\nvDash\psi_i^o(\bar{g}_o)$. Then $G^o\vDash\phi_I^o(\bar{g}_o)$ and 
         $H\vDash\phi_I^v(\bar{g}_v)$. Conversely, assume there is $I\in P$ with 
         $G^o\vDash\phi_I^o(\bar{g}_o)$ and 
         $H\vDash\phi_I^v(\bar{g}_v)$. Then for any $i\in\{1,\dots,k\}$, we either have 
         $G^o\nvDash\psi_i^o(\bar{g}_o)$ ( when $i\in I$) or $H\nvDash\psi_i^v(\bar{g}_v)$ (when $i\notin I$). By induction hypothesis, this means that 
         $G\nvDash\psi(\bar{g})$.

        \end{proof}

        An immediate consequence of Lemma \ref{fefermanvaughtlemma} is the following Theorem:

        \begin{Thm}\label{Fefermanvaughttheorem}
         Let $G_1^o,G_2^o$ be two ordered abelian groups and $H_1,H_2$ two valuationally quasi-ordered groups. 
           Then $G_1^o\cprod H_1\equiv G_2^o\cprod H_2$ if and only if 
           $G_1^o\equiv G_2^o$ and $H_1\equiv H_2$.
        \end{Thm}
        
        \begin{proof}
         One direction of the theorem is given by Proposition \ref{elementaryequivalenceofparts}, let us now prove 
         the converse. Assume that $G_1^o\equiv G_2^o$ and $H_1\equiv H_2$ holds and let $\phi$ be a sentence of $\langue$ with 
         $G_1^o\cprod H_1\vDash\phi$. Take $\phi_1^o,\dots,\phi_n^o,\phi_1^v,\dots,\phi_n^v$ 
         as in Lemma \ref{fefermanvaughtlemma}. Since $G_1^o\cprod H_1\vDash\phi$, there is 
         $i\in\{1,\dots,n\}$ such that $G_1^o\vDash\phi_i^o$ and $H_1\vDash\phi_i^v$. By assumption, we then have 
         $G_2^o\vDash\phi_i^o$ and $H_2\vDash\phi_i^v$, which by choice of $\phi_i^o,\phi_i^v$ implies that 
         $G_2^o\cprod H_2\vDash\phi$. We could show similarly that $G_2^o\cprod H_2\vDash\phi$ implies 
         $G_1^o\cprod H_1\vDash\phi$, hence $G_2^o\cprod H_2\equiv G_1^o\cprod H_1$.          
        \end{proof}

\section{Quasi-order-minimality and C-relations}\label{minimalitysection}

  The main goal of the theory of quasi-orders is to study orders and valuations in a common framework. 
  A powerful tool in the study of ordered structures is the notion of o-minimality. We now intend to develop an analogous
  notion for quasi-ordered groups called quasi-order-minimality.  O-minimality is a special case of 
  quasi-order-minimality and the latter should also be applicable to valued groups.
  
  An o-minimal group is defined as an ordered group $(G,\leq)$ such that any definable subset of $G$ is 
  a finite disjoint union of intervals. By analogy, we want to define a quasi-order-minimal group as a group in which every definable
  one-dimensional subset is a finite disjoint union of ``simple'' definable sets. This requires 
  first determining what the 
  ``simple definable sets'' are in the case of quasi-ordered groups.
  
  Jan Holly (see \cite{Holly}) already gave the shape of simple definable sets of valued fields: they are what she called 
  swiss cheeses, i.e
  sets of the form $X\backslash \bigcup_{i=1}^nX_i$ where $X$ and each $X_i$ is an ultrametric ball.
  Following her idea, we define a \textbf{$\precsim$-ball} of a compatible q.o.a.g $(G,\precsim)$ as a set of the form
  $\{g\in G\mid g-a\precsim b\}$ (closed ball) or $\{g\in G\mid g-a\precnsim b\}$(open ball) for some parameters
  $a,b\in G$. We then define a \textbf{swiss cheese} of $G$ as a subset of $G$ of the form
  $X\backslash \bigcup_{i=1}^nX_i$ where each $X_i$ and $X$ are $\precsim$-balls. Now consider compatible q.o.a.g's as structures
  of the language $\{0,+,-,\precsim\}$. We say that a compatible q.o.a.g
  $(G,0,+,-,\precsim)$ is 
  \textbf{quasi-order-minimal} if the following condition holds: 
  for every compatible q.o.a.g $(H,0,+,-,\precsim)$ which is elementarily equivalent to $(G,0,+,-,\precsim)$ 
  , every definable subset of $H$ is a finite disjoint union of swiss cheeses.
  Note that if $(G,\precsim)$ happens to be an ordered abelian group then the $\precsim$-balls are just
  initial segments, and the class of finite unions of swiss cheeses is exactly the class of finite unions of 
  intervals, so the notion of quasi-order-minimality coincides with o-minimality.

  As already announced in the introduction, we now want to show that compatible quasi-orders naturally induce a C-relation 
  and that quasi-order-minimality is equivalent to C-minimality.
  The notion of C-structure was first introduced by Adeleke and Neumann in \cite{Adeleke1} and \cite{Adeleke2}.
  Macpherson and Steinhorn then developed the notion of C-minimality and C-minimal groups in 
  \cite{Macstein}. Delon then gave a slightly more general definition of C-structures in \cite{Delon}, which is the 
  one we give here.
  A \textbf{C-relation} on a set $M$ is a ternary relation $C$ satisfying the following axioms:
	  \begin{itemize}
	\item[$(C_1)$] $C(x,y,z)\Rightarrow C(x,z,y)$
	\item[$(C_2)$] $C(x,y,z)\Rightarrow\neg C(y,x,z)$
	\item[$(C_3)$] $C(x,y,z)\Rightarrow C(w,y,z)\vee C(x,w,z)$
	\item[$(C_4)$] $x\neq y\Rightarrow C(x,y,y)$
	\end{itemize}
  We say that a structure $\mathcal{M}=(M,C,\dots)$ endowed with a C-relation is C-minimal if for every 
  $\mathcal{N}=(N,C,\dots)$ such that $\mathcal{N}\equiv\mathcal{M}$ every definable subset of $N$ is
  quantifier-free definable in the language $\{C\}$.
  If $G$ is a group and $C$ a C-relation on $G$, we say that $C$ is \textbf{compatible with $+$} if 
  $C(x,y,z)$ implies $C(v+x+u,v+y+u,v+z+u)$ for any $x,y,z,u,v\in G$ (Note that this is the definition 
  given in \cite{Macstein}, where the authors consider groups which are not necessarily abelian).  
  A C-group is a 
  pair $(G,C)$ consisting of a group $G$ with a C-relation $C$ compatible with $+$.
  
  We already know two examples of C-groups:
  \begin{enumerate}[(i)]
         \item If $(G,\leq)$ is an ordered abelian group, then 
       $\leq$ induces a C-relation compatible with $+$, defined by 
       $C(x,y,z)\Leftrightarrow (y<x\wedge z<x)\vee(y=z\neq x)$.
       \item If $(G,v)$ is a valued group, then $v$ induces a C-relation compatible with $+$
       by $C(x,y,z): v(y-z)>v(x-z)$.
             \end{enumerate}

             Since a compatible q.o is a mix of order and valuation, it is natural to wonder if it 
             induces a C-relation. As it happens, we can associate a compatible C-relation to any compatible 
             q.o.a.g in a canonical way. The idea is to mix the definition of a C-relation coming from an order 
             with the definition
             of a C-relation coming from a valuation.
             
       \begin{Prop}\label{Creldunsqo}
        Let $(G,\precsim)$ be a compatible q.o.a.g. Consider the relation 
        $C(x,y,z)$ defined by 
        \[(x\neq y=z)\vee(x-z\in G^v\wedge (y-z\precnsim x-z))
    \vee(y-z,x-z\in G^o
   \wedge (0\precnsim x-y\wedge 0\precnsim x-z))\]
         This is a C-relation compatible with $+$, called the C-relation induced by the q.o $\precsim$.
         Moreover, $\precsim$ is the only compatible q.o inducing $C$, $C$ is quantifier-free definable in the language
         $(0,+,-,\precsim)$ and $\precsim$ is quantifier-free definable in $(0,+,-,C)$.          
       \end{Prop}

        \begin{proof}
          Let $v$ be the valuation defined in Proposition \ref{valuation} and $C^v$ the C-relation it induces, i.e  
        $C^v(x,y,z)\Leftrightarrow v(y-z)> v(x-z)$.         
        Let us show that $C$ satisfies the axioms of C-relations. It is clear from the definition of $C$ that
         $(C_4)$ holds. We prove $(C_1),(C_2),(C_3)$ simultaneously; 
         take $x,y,z,w\in G$ such that 
         $C(x,y,z)$ and let us show that
         $C(x,z,y),\neg C(y,x,z)$ and 
         $C(w,y,z)\vee C(x,w,z)$ hold.
         Assume first that $y=z$. Then clearly $C(x,z,y)$ holds.     Note also that by definition of $C$, $C(y,x,y)$ 
         cannot hold, so $\neg C(y,x,z)$ hods. Finally, if $w\neq y$ then by $(C_4)$ $C(w,y,z)$ holds, and if 
         $w=y$ then $C(x,w,z)$ holds. Now assume
         $z\neq y$.
         Assume first that $x-z\in G^v$. We thus have $C^v(x,y,z)$ which implies 
         $C^v(x,z,y)$ which means $v(z-y)>v(x-y)$ and since $v$ takes its maximal non-infinite value on $G^o\backslash\{0\}$ and is constant 
         on $G^o\backslash\{0\}$, it follows that 
         $x-y\in G^v$, 
         so $C(x,z,y)$ holds.
         We also have $\neg C^v(y,x,z)$, and since $x-z\notin G^o$ this implies $\neg C(y,x,z)$.
         If $x-z\precsim w-z$ then $w-z\in G^v$ because $G^v$ is a final segment of $G$. In that case we have 
         $C(w,y,z)$. Otherwise we have $w-z\precnsim x-z$ so $C(x,w,z)$.
         Assume now that $x-z,y-z\in G^o$, so $0\precnsim x-y\wedge 0\precnsim x-z$. We can obviously exchange
         $x$ and $y$ in this formula, hence $C(x,z,y)$. However, $0\precnsim x-y\in G^o$ implies 
         $y-x\precnsim 0$, so $C(y,x,z)$ does not hold. If $w-z\in G^v$ then since $y-z\in G^o$ we have 
         $y-z\precnsim w-z$ hence $C(w,y,z)$. Assume $w-z\in G^o$, which also implies $w-x,w-y\in G^o$. If 
         $C(w,y,z)$ does not hold, then either $w-z\precnsim 0$ or $w-y\precnsim 0$ must be true. If
         $w-z\precnsim 0$ then $0\precnsim z-w$, hence $0\precnsim x-w=x-z+z-w$, so $C(x,w,z)$ holds; the same reasoning
         holds if $w-y\precnsim 0$.

          The fact that $C$ is compatible with $+$ is obvious from its definition.
          Note that $G^o$, $G^v$ are both quantifier-free definable in the language $\{0,+,-,\precsim\}$ since we have
          $G^o=\{g\in G\mid g\nsim -g\vee g=0\}$ and $G^v=\{g\in G\mid g\sim -g\neq0\}$, so $C$ is defined with a 
          quantifier-free formula
          of that language.
          We want to show the converse. Set $G^+:=\{g\in G\mid 0\precsim g\}$ and $G^-=\{g\in G\mid g\precnsim 0\}$.
          We want to find a formula defining $\precsim$ in the language $\{0,+,-,C\}$.
          Note first that we have
          \[x\precnsim y\Leftrightarrow (x\in G^-\wedge y\in G^+)\vee(x,y\in G^+\wedge x\precnsim y)\vee
          (x,y\in G^-\wedge -y\precnsim -x)\]
          It is easy to see from the definition of $C$ that for any $x,y\in G^+$, 
          $x\precnsim y\Leftrightarrow C(y,x,0)$. Moreover, $G^-$ and $G^+$ are quantifier-free 
          definable with $C$: 
          $G^+\cap G^o$ is given by the formula $C(x,-x,0)\vee(x=0)$. Indeed,  by definition of $C$ we have
          $C(x,-x,0)\Leftrightarrow (x\neq -x=0)\vee(x\in G^v\wedge -x\precnsim x)\vee (x,-x\in G^o\wedge 0\precnsim x\wedge
          -x\precnsim x)$. Obviously $(x\neq -x=0)\vee(x\in G^v\wedge -x\precnsim x)$ is impossible 
          (if $x$ is in $G^v$ then $x\sim -x$) so $C(x,-x,0)\Leftrightarrow(x,-x\in G^o\wedge 0\precnsim x\wedge
          -x\precnsim x)$ which means $x\in G^o\cap G^+$. It follows that 
          $G^-$ is defined by the formula $C(-x,x,0)$ and that $G^v$ is defined by $\neg C(x,-x,0)\wedge\neg C(-x,x,0)$.
          Thus, the formula
          \[\phi(x,y):\equiv(x\in G^-\wedge y\in G^+)\vee(x,y\in G^+\wedge C(y,x,0))\vee
          (x,y\in G^-\wedge C(-x,-y,0))\] is a quantifier-free formula of the language $\{0,+,-,C\}$ and we have 
          $x\precsim y\Leftrightarrow\neg \phi(y,x)$ for any $x,y\in G$. This proves that $\precsim$ is 
          quantifier-free definable
          in $\{0,+,-,C\}$ and it also proves that $\precsim$ is the only compatible q.o inducing $C$ since we can 
          recover $\precsim$ from $C$.

        \end{proof}

        Thus, compatible q.o.a.g's can be seen as C-groups. However, they do not describe every abelian C-group:
        we can give examples
        of abelian C-groups whose C-relation does not come from a compatible q.o.
        Consider $G:=\Z^2$ endowed with the following C-relation:
        \[C(x,y,z)\Leftrightarrow 
        (x_1\neq y_1=z_1)\vee(y_1<x_1\wedge z_1<x_1)\vee (x\neq y=z)\]
        where $x=(x_1,x_2), y=(y_1,y_2)$ and $z=(z_1,z_2)$.
        It is easy to see that this is a compatible C-relation, but it is not induced by a compatible q.o.
        This proves that the class of compatible quasi-orders does not coincide with the whole class of abelian C-groups; however, it will be showed
        in \cite{Lehericy} that we can still use quasi-orders to classify 
        C-groups. More precisely, \cite{Lehericy} introduces the notion of C-quasi-orders, which in some sense generalize compatible 
        q.o's. C-quasi-orders are in a one-to-one correspondence with compatible C-relations on groups. In \cite{Lehericy}, we use 
        C-quasi-orders to describe the structure of C-groups. The main theorem of \cite{Lehericy} (Theorem 3.41) states that 
        C-quasi-ordered groups are also a ``mix'' of ordered and valued groups, but the ``mix'' is in general more complicated than 
        in the case of compatible q.o's; indeed, a C-quasi-order can alternate infinitely many times between ``ordered'' parts and 
        ``valued'' parts, whereas the ordered part of a compatible q.o is always an initial segment of the group. This means 
        that compatible q.o.a.g's form a particularly simple class of C-groups and makes them practical for the study of 
        C-groups.
        
        We now show that quasi-order-minimality is equivalent to 
        C-minimality. In \cite{Delon} the author defined a notion of swiss cheeses for C-structures.
         Let $(M,C,\dots)$ be a C-structure.
         A \textbf{cone} is a subset of $M$ of the form $\{x\mid C(a,x,b)\}$ for some parameters $a,b\in M$.
         A \textbf{thick cone} is a subset of $M$ of the form $\{x\mid \neg C(x,a,b)\}$ for some parameters $a,b\in M$.
         A \textbf{C-swiss cheese} is a subset of $M$ of the form $X\backslash\bigcup_{i=1}^nX_i$ where each 
         $X_i$ and $X$ 
         is a (possibly thick) cone.         
        The author of \cite{Delon} then gives a characterization of C-minimal structures in terms of $C$-swiss cheeses (Proposition 
        3.3 of \cite{Delon}) which we reformulate here:
        
        \begin{Prop}
         Let $\mathcal{M}=(M,C,\dots)$ be a C-structure. Then $(M,C,\dots)$ is C-minimal if and only if for every 
         $\mathcal{N}=(N,C,\dots)$ with $\mathcal{N}\equiv\mathcal{M}$ every definable subset of 
         $N$ is a finite disjoint
         union of $C$-swiss cheeses.
        \end{Prop}

        If we have a compatible q.o.a.g $(G,\precsim)$ and $C$ is the C-relation induced by $\precsim$, it is easy to see
        that the cones (respectively the thick cones) of the C-structure are exactly the open 
        (respectively closed) $\precsim$-balls. It follows that the $C$-swiss cheeses of the C-structure coincide with 
        the swiss cheeses of the quasi-order structure and that quasi-order-minimality 
        is equivalent to C-minimality:
        
        \begin{Prop}\label{Cminimal}
          Let $\mathcal{L}$ be a language containing $\{0,+,-,\precsim\}$ and let $\mathcal{L}'$ be the 
          language obtained by replacing $\precsim$ by $C$ in $\mathcal{L}$.
         Assume $(G,\precsim,\dots)$ is an $\mathcal{L}$-structure so that $(G,\precsim)$
         is a compatible q.o.a.g and let $(G,C,\dots)$ be the corresponding $\mathcal{L}'$-structure where $C$ is
         interpreted as the C-relation induced by $\precsim$. Then
         $(G,\precsim,\dots)$ is quasi-order-minimal if and only if  $(G,C,\dots)$ is C-minimal.
        \end{Prop}

       Finally, as a motivation for the study of compatible q.o.a.g's, we give an example of a C-minimal group  whose C-relation comes from a compatible q.o which 
       is neither an order nor a valuation. Theorem \ref{prodCminimal} will allow us to construct such examples. We first need to show the following lemma:

       \begin{Lem}\label{ultrballvpart}
        Let $G^o$ be an ordered group,
         $H$ a valuationally quasi-ordered group such that $H\backslash\{0\}$ has
         a minimum $m$ and set $G:=G^o\cprod H$. Then for any boolean combination $\phi(x,\bar{a}_v)$ of balls 
         with $\bar{a}_v\subseteq H$, there exists a boolean combination $\phi^{\ast}(x,\bar{a}_v,m)$ of balls 
         with $\bar{a}\subseteq G$ such that for any 
        $g=g_o+g_v\in G$, $G\vDash\phi^{\ast}(g,\bar{a}_v,m)$ if and only if $H\vDash\phi(g_v,\bar{a}_v)$.
       \end{Lem}
	\begin{proof}
	 It is sufficient to show the lemma in the case where $\phi$ is a ball. 
	 Assume then that $\phi(x)\equiv x-a_v\precsim b_v$. If $b_v\neq0$, then by Remark \ref{remarkonquotientqo}(4)
	 $g_v-a_v\precsim b_v$ is equivalent to $g-a_v\precsim b_v$ for any $g\in G$, so we can just set 
	 $\phi^{\ast}(x)\equiv  x-a_v\precsim b_v$. If $b=0$, then $g_v-a_v\precsim b_v$ is true if and only if 
	 $g-a_v\in G^o$ which is equivalent to $g-a_v\precnsim m$, so we can set 
	 $\phi^{\ast}(x)\equiv x-a_v\precnsim m$. Assume now that $\phi(x)\equiv x-a_v\precnsim b_v$. 
	 If $b_v=0$ then $\phi(x)$ is not satisfiable in $H$, so we set $\phi^{\ast}(x)\equiv x\precnsim 0\wedge 0\precnsim x$. If 
	 $b_v\neq0$, we can set $\phi^{\ast}(x)\equiv x-a_v\precnsim b_v$ thanks to Remark \ref{remarkonquotientqo}(4).
	\end{proof}

       \begin{Thm}\label{prodCminimal}
        Let $G^o$ be an o-minimal group and $H$ a finite valuationally quasi-ordered group. 
        Then $G:=G^o\cprod H$ is C-minimal.
       \end{Thm}

       \begin{proof}
       Note that any finite q.o.a.g is C-minimal, so in particular $H$ is C-minimal.
       Moreover, $H\backslash\{0\}$ admits a minimum which we will denote by $m$. For any $g\in G$, we denote 
       by $g_o$ and $g_v$ the unique elements of $G^o$ and $H$ such that $g=g_o+g_v$.
       We first show that any definable subset of $G$ is a boolean combination of balls.
       Let $\phi(x,\bar{a})$ be a formula of $\langue$ with one free variable and parameters $\bar{a}\subseteq G$. Take 
       $\phi_1^o(x,\bar{y}),\dots,\phi_n^o(x,\bar{y}),\phi_1^v(x,\bar{y}),\dots,\phi_n^v(x,\bar{y})$ as in Lemma \ref{fefermanvaughtlemma}. 
       By Lemma \ref{fefermanvaughtlemma}, it is sufficient to show that 
       for each $i\in\{1,\dots, n\}$, the set 
       $A_i:=\{g\in G\mid G^o\vDash\phi_i^o(g_o,\bar{a}_o)\wedge H\vDash\phi_i^v(g_v,\bar{a}_v)\}$ is a boolean combination 
       of  balls. Fix an $i\in\{1,\dots,n\}$. Since $G^o$ and $H$ are C-minimal, 
       there are two formulas $\theta_o(x,\bar{b}_o),\theta_v(x,\bar{b}_v)$  (with parameters $\bar{b}_o\subseteq G^o$ and 
       $\bar{b}_v\subseteq H$) which are boolean combinations of  balls such that 
       $G^o\vDash\forall x(\phi_i^o(x,\bar{a}_o)\Leftrightarrow \theta_o(x,\bar{b}_o))$ and 
       $H\vDash (\phi_i^v(x,\bar{a}_v)\Leftrightarrow\theta_v(x,\bar{b}_v))$. For any $g=g_o+g_v\in G$, we now have 
       $g\in A_i$ if and only if $G^o\vDash\theta_o(g_o,\bar{b}_o)$ and $H\vDash\theta_v(g_v,\bar{b}_v)$. 
       Now note that  $G^o\vDash\theta_o(g_o,\bar{b}_o)$ if and only if 
       $G\vDash\bigwedge_{h\in H}(g-h\precnsim m\Rightarrow \theta_o(g-h,\bar{b}_o))$ (this follows from the fact that 
       $g_o=g-h$ if and only if $g-h\in G^o$ if and only if $g-h\precnsim m$), and that the formula 
       $\bigwedge_{h\in H}(x-h\precnsim m\Rightarrow \theta_o(x-h,\bar{b}_o))$
       is still a boolean combination of  balls.
       Finally, take $\theta_v^{\ast}(x,\bar{b}_v,m)$ as given by Lemma \ref{ultrballvpart}. The formula 
       $\theta_v^{\ast}(x,\bar{b}_v,m)\wedge\bigwedge_{h\in H}(x-h\precnsim m\Rightarrow \theta_o(x-h,\bar{b}_o))$ is a
       boolean combination of balls which defines $A_i$. This proves that every definable subsets of $G$ is a boolean 
       combination of balls, from which it easily follows that every definable subset of $G$ is a finite union of swiss cheeses. 
       Now  we must  show that the same is true for $G_2$, where $G_2$ is an arbitrary compatible 
       q.o.a.g with $G_2\equiv G$. By Proposition \ref{elementaryequivalenceofparts}, we have 
       $G_2^o\equiv G^o$ and $H_2\equiv H$, where $H_2$ denotes the valued part of $G$. It follows that 
       $G_2^o$ is o-minimal and that $H_2$ is C-minimal and finite (of the same cardinality as $H$). 
       Since $G_2^o$ is o-minimal, it is divisible, so $G_2^o$ is a direct summand of $G_2$, so by Proposition 
       \ref{Gascompatibleproduct} we have $G_2=G_2^o\cprod H_2$. It then follows from what we have already proved that every  
       definable subset of $G_2$ is a finite union of swiss cheeses.
       \end{proof}
       
       \begin{Rem}
        The condition of $H$ being finite was essential in the proof of Theorem \ref{prodCminimal}. 
        Indeed, the assumption that $G^o$ is C-minimal only tells us 
        that the set  $B$ of all $g_o$'s such that $g\in A_i$ is a boolean combination of balls, so it gives us a formula 
        $\theta_o(g_o,\bar{b}_o)$ in which $g_o$ appears. We then need to characterize the $g$'s 
        of $G$ such that $g_o\in B$ with an appropriate formula, i.e we need to
        ``lift'' $\theta_o(g_o,\bar{b}_o)$ to a formula in which $g$ appears instead of $g_o$.         
        The problem is that $g_o$ is in general not definable in 
        $G$ if $H$ is chosen arbitrarily, so we cannot express ``$g_o\in B$'' with a formula. 
        However, if we happen to know that the $g_v$'s of all $g$'s in $A_i$  only take 
        finitely many values in $H$ (as is the case in Theorem \ref{prodCminimal}), then we 
        can express ``$g_o\in B$'' via a formula 
        $\bigwedge_h(g-h\precnsim m\Rightarrow \theta_o(g-h,\bar{b}_o))$, where 
        $h$ ranges over all possible values of $g_v$ for $g$ in $A_i$.  
       \end{Rem}

       We can now give an example of a C-minimal group which is neither ordered nor valued:
       \begin{Ex}
        Let $G^o:=\Q$ with the usual order; it is known that this is an o-minimal structure. 
        Set \newline        
        $H:=(\Z/p^k\Z,v_p)$ with $k\in\N$, where $v_p$ denotes the valuation induced on $H$ by the p-adic valuation of 
        $\Z$. Then by Theorem \ref{prodCminimal}, $G^o\cprod H$ is C-minimal.
       \end{Ex}

 \bibliographystyle{plain}
 \bibliography{referencesASTFAQOG.bib}

\end{document}